%% file: epidemic-3.tex
\documentclass[letterpaper,11pt,twoside,keywordsasfootnote,addressatend,noinfoline]{article}
\usepackage{fullpage}
\usepackage[english]{babel}
\usepackage{amssymb}
\usepackage{amsmath}
\usepackage{comment}
\usepackage{theorem}
\usepackage{epsfig}
\usepackage{subfigure}
\usepackage{mathrsfs}
\usepackage{color}
\usepackage{imsart}

\newtheorem{theorem}{Theorem}
\newtheorem{lemma}[theorem]{Lemma}

\newenvironment{proof}{\noindent{\bf Proof.}}{\hspace*{2mm}~$\square$}

\newcommand{\Z}{\mathbb{Z}}
\newcommand{\R}{\mathbb{R}}

\newcommand{\Lat}{\mathscr{L}}

\newcommand{\norm}[1]{|\!|#1|\!|}
\newcommand{\ind}{\mathbf{1}}
\newcommand{\ep}{\epsilon}
\newcommand{\n}{\hspace*{-5pt}}

\DeclareMathOperator{\tr}{Tr}
\DeclareMathOperator{\de}{Det}
\DeclareMathOperator{\vol}{Vol}

\DeclareMathOperator{\poisson}{Poisson}

\DeclareMathOperator{\geometric}{Geometric}
\DeclareMathOperator{\exponential}{Exponential}

\begin{document}

\begin{frontmatter}
\title     {The contact process with an asymptomatic state}
\runtitle  {The contact process with an asymptomatic state}
\author    {Lamia Belhadji, Nicolas Lanchier\thanks{Nicolas Lanchier was partially supported by NSF grant CNS-2000792.} and Max Mercer}
\runauthor {Lamia Belhadji, Nicolas Lanchier and Max Mercer}
\address   {Faculty of Exact Sciences \\ and Computer Science \\ University of Abdelhamid Ibn Badis Mostaganem \\ Algeria \\ lamia.belhadji@univ-mosta.dz}
\address   {School of Mathematical \\ and Statistical Sciences \\ Arizona State University \\ Tempe, AZ 85287, USA. \\ nicolas.lanchier@asu.edu \\ mamerce1@asu.edu}

\maketitle

\begin{abstract} \ \
 In order to understand the cost of a potentially high infectiousness of symptomatic individuals or, on the contrary, the benefit of social distancing, quarantine, etc. in the course of an infectious disease, this paper considers a natural variant of the popular contact process that distinguishes between asymptomatic and symptomatic individuals.
 Infected individuals all recover at rate one but infect nearby individuals at a rate that depends on whether they show the symptoms of the disease or not.
 Newly infected individuals are always asymptomatic and may or may not show the symptoms before they recover.
 The analysis of the corresponding mean-field model reveals that, in the absence of local interactions, regardless of the rate at which asymptomatic individuals become symptomatic, there is an epidemic whenever at least one of the infection rates is sufficiently large.
 In contrast, our analysis of the interacting particle system shows that, when the rate at which asymptomatic individuals become symptomatic is small and the asymptomatic individuals are not infectious, there cannot be an epidemic even when the symptomatic individuals are highly infectious.
\end{abstract}



\end{frontmatter}

%%%%%%%%%%%%%%%%%%%%%%%%%%%%%%%%%%%%%%%%%%%%%%%%%%%%%%%%%%%%%%%%%%%%%%%%%%%%%%%%%%%%%%%%%%%%%%%%%%%%%%%%%%%%%%%%%%%%%%%%%%%%%%%%%%%%%%%%%%%%%%%%%%%%%%%%%%%%%%%%%%%%%%%%%%%%%%%%%%%%%%%%%%%%%%%%%%

\section{Introduction}
\label{sec:intro}
 Infectious diseases commonly spread through the transfer of bacteria from one person to another either directly through talking, coughing, sneezing, etc. or indirectly through touching and subsequently infecting objects which are then touched by another person.
 It is also common that infected individuals do not immediately exhibit the symptoms of the disease, which leads to two categories of infected individuals that we call asymptomatic individuals~(the ones who do not exhibit the symptoms yet) and symptomatic individuals~(the ones who exhibit the symptoms).
 Typically, asymptomatic individuals are less contagious but unaware that they have contracted the disease, whereas symptomatic individuals are more contagious but can take precautions~(use of a mask, social distancing, quarantine, use of a condom, etc.) to limit the spread of the disease.
 In particular, the characteristics of the disease and the social behavior of the population induce a variability in the rate at which the disease spreads from asymptomatic versus symptomatic individuals. \\
\indent
 To understand the cost of a potentially high infectiousness of symptomatic individuals due to increased spread of bacteria via frequent coughing, sneezing, etc., or on the contrary, the benefit of social distancing, quarantine, etc., we study a natural variant of Harris'~\cite{Harris_1974} contact process in which infected individuals are divided into two categories: asymptomatic individuals and symptomatic individuals.
 Like in the contact process, individuals are located on the~$d$-dimensional integer lattice, and infected individuals infect their healthy neighbors by contact, and recover at rate one.
 However, the rate at which infected individuals infect their neighbors now depends on whether they are asymptomatic or symptomatic, with newly infected individuals being always asymptomatic before possibly transitioning to the symptomatic state.
 More precisely, each site is in state
 $$ 0 = \hbox{healthy}, \qquad 1 = \hbox{infected/asymptomatic}, \qquad \hbox{or} \qquad 2 = \hbox{infected/symptomatic}, $$
 so the state of the entire system at time~$t$ is a spatial configuration
 $$ \xi_t : \Z^d \longrightarrow \{0, 1, 2 \} \quad \hbox{where} \quad \xi_t (x) = \hbox{state at site~$x$ at time~$t$}. $$
 We assume for simplicity that individuals can only interact with their nearest neighbors but most of our results easily extend to larger ranges of interactions.
 To describe the dynamics~(local interactions) in the nearest-neighbor case, we let
 $$ N_x = \{y \in \Z^d : \norm{x - y} = 1 \} \quad \hbox{for all} \quad x \in \Z^d, $$
 where~$\norm{\cdot}$ refers to the Euclidean norm, be the set of nearest neighbors of site~$x$.
 Then, letting for each type~$i$ of infected individuals, each site~$x$ and each configuration~$\xi$,
 $$ \begin{array}{l} f_i (x, \xi) = (1/2d) \,\sum_{y \in N_x} \ind \{\xi (y) = i \} \end{array} $$
 be the fraction of neighbors of site~$x$ that are in state~$i$, the contact process with an asymptomatic state evolves according to the local transition rates
 $$ \begin{array}{rclcrcl}
      0 \to 1 & \hbox{at rate} & \beta_1 f_1 (x, \xi) + \beta_2 f_2 (x, \xi), & \quad & 1 \to 2 & \hbox{at rate} & \gamma, \vspace*{4pt} \\
      1 \to 0 & \hbox{at rate} & 1, & \quad & 2 \to 0 & \hbox{at rate} & 1. \end{array} $$
 The first transition rate indicates that healthy individuals get infected from nearby asymptomatic individuals at rate~$\beta_1$ and from nearby symptomatic individuals at rate~$\beta_2$.
 Though the infection rate~$\beta_2$ is expected to be larger than~$\beta_1$, factors such as social distancing and quarantine may contribute to make~$\beta_2$ smaller than the infection rate~$\beta_1$.
 The two transition rates at the bottom indicate that infected individuals recover at rate one regardless of whether they are symptomatic or not.
 Finally, the second transition rate at the top indicates that asymptomatic individuals may become symptomatic before they recover.
 For small~$\gamma$, most of the infected individuals will recover before showing any symptoms and may be unaware that they were ever infected, whereas, for large~$\gamma$, most of the infected individuals will quickly show the symptoms of the disease. \\
\indent
 Our spatial stochastic epidemic model is most closely related to Krone's~\cite{Krone_1999} two-stage contact process and Deshayes'~\cite{Deshayes_2014} contact process with aging.
 In these two models, state~0 is interpreted as an empty site, while positive states represent the age of an individual.
 The two-stage contact process includes only two age classes, namely,~1~=~juvenile and~2~=~adult.
 Adults give birth onto nearby empty sites to a juvenile at rate~$\beta_2$ and die at rate one, while juveniles cannot give birth, transition to the adult stage at rate~$\gamma$, and die at a possibly higher rate~$1 + \delta$.
 In particular, the two-stage contact process with~$\delta = 0$ is equivalent to our epidemic model with~$\beta_1 = 0$ in which asymptomatic individuals are not infectious.
 In the contact process with aging, the individuals' age is indexed by the positive integers.
 Individual with age~$i$ give birth onto nearby empty sites to an individual with age~1 at rate~$\beta_i$ with the condition~$\beta_1 < \beta_2 < \beta_3 < \cdots$, die at rate one regardless of their age, and transition to the next age class at rate~$\gamma$.
 In particular, the contact process with aging truncated at age~2 is equivalent to our epidemic model with~$\beta_1 < \beta_2$ in which the asymptomatic individuals spread the disease at a slower rate than the symptomatic individuals. \\
\indent
 Assuming that the population is homogeneously mixing, and letting~$u_1$ be the density of infected individuals who are asymptomatic and~$u_2$ be the density of infected individuals who are symptomatic, the system is described by the mean-field model
 $$ u_1' = (\beta_1 u_1 + \beta_2 u_2)(1 - u_1 - u_2) - (1 + \gamma) u_1 \quad \hbox{and} \quad u_2' = \gamma u_1 - u_2. $$
 In this context, we say that the infection survives when
\begin{equation}
\label{eq:mean-field-survival}
\begin{array}{rcl} (u_1 + u_2)(0) > 0 & \Longrightarrow & \lim_{t \to \infty} (u_1 + u_2)(t) > 0. \end{array}
\end{equation}
 The analysis of the mean-field model in Section~\ref{sec:mean-field} shows that, in addition to the disease-free equilibrium~$p_0 = (0, 0)$, there is a nontrivial fixed point~$p_{12}$ if and only if~$\beta_1 + \gamma \beta_2 > 1 + \gamma$, and that this fixed point is globally stable.
 In particular, for all~$0 < \gamma < \infty$, even if one of the two infection rates is equal to zero, the infection survives whenever the other infection rate is large.
 Our analysis of the spatial model again shows that, for all~$0 < \gamma < \infty$ and all~$\beta_2$, the infection survives provided the infection rate~$\beta_1$ is sufficiently large.
 However, when~$\gamma$ is small and~$\beta_1 = 0$, the infection dies out regardless of the infection rate~$\beta_2$, which strongly contrasts with the behavior of the mean-field model.
 For the interacting particle system, we distinguish two types of survival/extinction.
 Starting from a translation-invariant distribution with infinitely many~1s and~2s, we say that the infection survives or dies out locally based on the limiting density of infected individuals:
 $$ \begin{array}{c}
    \hbox{local survival if} \ \lim_{t \to \infty} P (\xi_t (0) \neq 0) > 0, \quad \hbox{local extinction if} \ \lim_{t \to \infty} P (\xi_t (0) \neq 0) = 0. \end{array} $$
 In contrast, when starting from a single infected individual at the origin, we say that the infection survives or dies out globally based on the number of infected individuals in the system: identifying the process with the set of infected sites, we have
 $$ \begin{array}{c}
    \hbox{global survival if} \ P (\xi_t \neq \varnothing \ \hbox{for all} \ t) > 0, \quad \hbox{global extinction if} \ P (\xi_t \neq \varnothing \ \hbox{for all} \ t) = 0. \end{array} $$
 Intuitively, the larger the infection rates, the larger the infected region:
\begin{equation}
\label{eq:betas-monotone}
\begin{array}{l}
  P (\xi_t (0) \neq 0) \ \hbox{is nondecreasing with respect to~$\beta_1$ and~$\beta_2$}. \end{array}
\end{equation}
 In addition, because increasing~$\gamma$ decreases the density of asymptomatic individuals but increases the density of symptomatic individuals, we also expect that
\begin{equation}
\label{eq:gamma-monotone}
\begin{array}{rcl}
\beta_1 < \beta_2 & \Longrightarrow & P (\xi_t (0) \neq 0) \ \hbox{is nondecreasing with respect to~$\gamma$}, \vspace*{4pt} \\
\beta_1 > \beta_2 & \Longrightarrow & P (\xi_t (0) \neq 0) \ \hbox{is nonincreasing with respect to~$\gamma$}. \end{array}
\end{equation}
 Coupling processes with different parameters shows~\eqref{eq:betas-monotone} and~\eqref{eq:gamma-monotone} when~$\beta_1 < \beta_2$.
\begin{figure}[t!]
\centering
\scalebox{0.40}{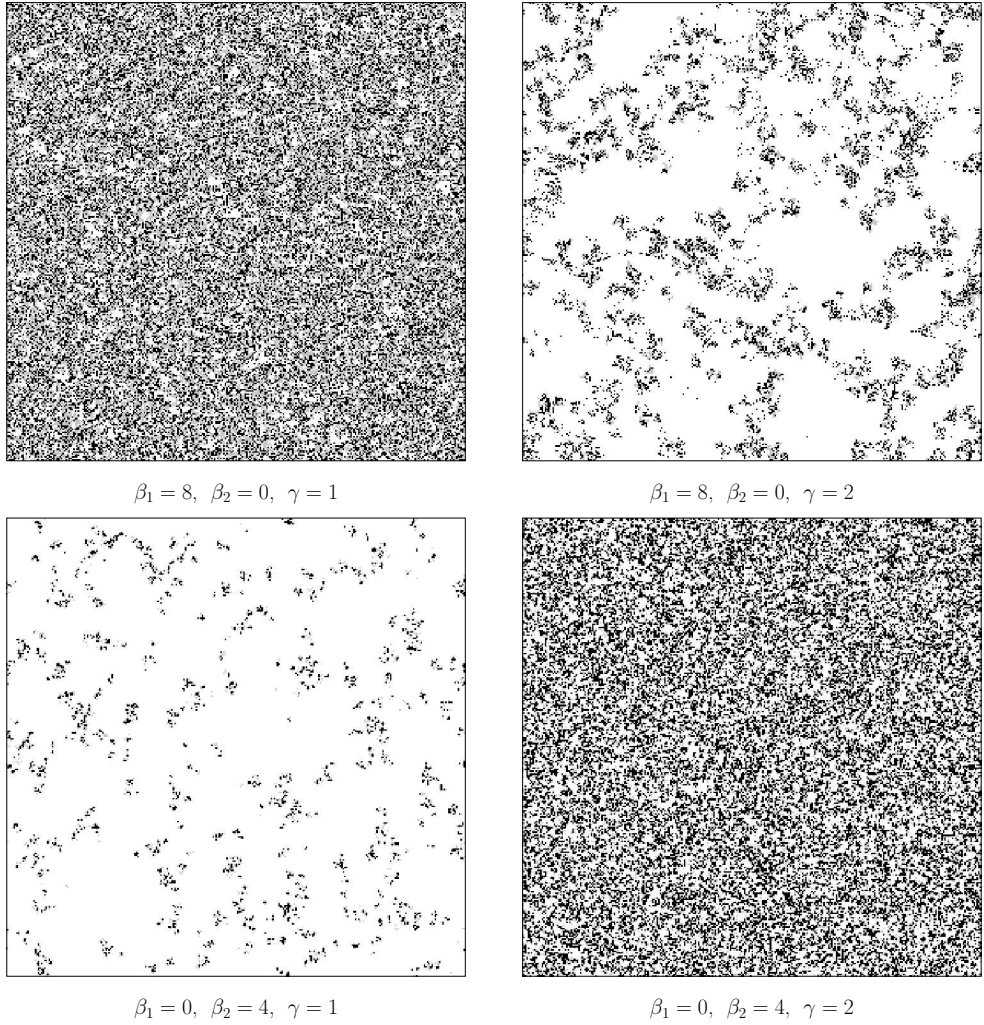}
\caption{\upshape{
 Snapshots of the process on the~$300 \times 300$ torus at time~200.
 In all the pictures, the white sites represent the individuals who are healthy, the gray sites the individuals who are infected but asymptomatic, and the black sites the individuals who are infected and symptomatic.
 The simulations illustrate the reversed monotonicity with respect to the parameter~$\gamma$ depending on the order of the two infection rates~$\beta_1$ and~$\beta_2$.}}
\label{fig:IPS}
\end{figure}
 Though we believe that these results also hold when~$\beta_1 > \beta_2$, which is supported by numerical simulations of the process~(see~Figure~\ref{fig:IPS}), similar couplings fail in this case.
 Monotonicity implies that, two of the parameters being fixed, there is at most one phase transition~(from local/global extinction to local/global survival) in the direction of the third parameter.
 To prove the existence of phase transitions, we first compare the process with the basic contact process~\cite{Harris_1974}, the system in which sites can be in state~0~=~healthy or in state~1~=~infected, with transition rates
 $$ \begin{array}{rclcrcl}
      0 \to 1 & \hbox{at rate} & \beta f_1 (x, \xi) & \quad \hbox{and} \quad & 1 \to 0 & \hbox{at rate} & 1. \end{array} $$
 The main result about the contact process states that there exists a nondegenerate~(positive and finite) critical value~$\beta_c$ that depends on the spatial dimension such that
\begin{equation}
\label{eq:survival}
\begin{array}{rcl} \hbox{the infection survives locally/globally} & \Longleftrightarrow & \beta > \beta_c. \end{array}
\end{equation}
 In particular, the infection dies out at the critical value~$\beta_c$, which was proved by Bezuidenhout and~Grimmett~\cite{Bezuidenhout_Grimmett_1990}.
 For a review of this process, see, e.g., Liggett~\cite{Liggett_1985,Liggett_1999} and Lanchier~\cite{Lanchier_2024}.
 Our epidemic model can be viewed as a mixture of the contact process with parameter~$\beta_1$ and the contact process with parameter~$\beta_2$, with the relative weight of the second contact process being measured by the rate~$\gamma$ at which asymptomatic individuals become symptomatic.
 Observing that the rate at which infected individuals infect their neighbors is between~$\beta_1 \wedge \beta_2$ and~$\beta_1 \vee \beta_2$, and that the rate at which they recover is equal to one regardless of their state~(asymptomatic or symptomatic), it follows from~\eqref{eq:survival} and a standard coupling argument that
\begin{equation}
\label{eq:contact}
\begin{array}{rcl}
\hbox{a.} \ \ \beta_1 \leq \beta_c \quad \hbox{and} \quad \beta_2 \leq \beta_c & \Longrightarrow & \hbox{local and global extinction}, \vspace*{4pt} \\
\hbox{b.} \ \ \beta_1 > \beta_c \quad \hbox{and} \quad \beta_2 > \beta_c & \Longrightarrow & \hbox{local and global survival}. \end{array}
\end{equation}
 We now look at the parameter regions where one of the two infection rates is subcritical and the other one supercritical, in which case the behavior of the process is less obvious. \\ \\
 {\bf Asymptomatic = subcritical and symptomatic = supercritical} \vspace*{5pt} \\
 To begin with, note that taking~$\gamma = 0$ blocks the access to the symptomatic state, so the process reduces to the basic contact process with parameter~$\beta_1$.
 On the contrary, taking~$\gamma = \infty$ makes the transition~$1 \to 2$ instantaneous, thus removing the asymptomatic state, so the process reduces to the basic contact process with parameter~$\beta_2$.
 Using these observations together with block constructions and perturbation techniques, we will prove that
\begin{theorem}
\label{th:top-left}
 Starting with a positive density of~1s and~2s,
\begin{equation}
\label{eq:top-left}
\begin{array}{rcl}
\beta_1 < \beta_c \quad \hbox{and} \quad \beta_2 > \beta_c & \Longrightarrow & \left\{\begin{array}{l}
\hbox{a. local survival for all~$\gamma < \infty$ large}, \vspace*{4pt} \\
\hbox{b. local extinction for all~$\gamma > 0$ small}. \end{array} \right. \end{array}
\end{equation}
\end{theorem}
 We now look at the~$\beta_1 = 0$ case in which the asymptomatic individuals cannot infect their healthy neighbors.
 In this context, the transition rates of the epidemic model reduce to
\begin{equation}
\label{eq:two-stage}
\begin{array}{rclcrcl}
  0 \to 1 & \hbox{at rate} & \beta_2 f_2 (x, \xi), & \quad & 1 \to 2 & \hbox{at rate} & \gamma, \vspace*{4pt} \\
  1 \to 0 & \hbox{at rate} & 1, & \quad & 2 \to 0 & \hbox{at rate} & 1. \end{array}
\end{equation}
 As mentioned earlier, this process is the particular case of Krone's~\cite{Krone_1999} two-stage contact process with parameter~$\delta = 0$.
 Using an idea due to Foxall~\cite{Foxall_2015} to compare the epidemic model with a certain subcritical Galton--Watson branching process, we will prove that
\begin{theorem}
\label{th:two-stage}
 Starting with a positive density of~1s and~2s,
 $$ \begin{array}{rcl} \gamma < 1 / (4d - 1) \quad \hbox{and} \quad \beta_1 = 0 & \Longrightarrow & \hbox{local extinction for all \,$\beta_2 \geq 0$}. \end{array} $$
\end{theorem}
 Letting~$\beta_1 = 0$ and~$\beta_+ > \beta_c$ be fixed, it follows from Theorem~\ref{th:top-left} and the monotonicity with respect to the parameters that there exists~$\gamma_+ = \gamma_+ (\beta_+) < \infty$ such that local survival occurs when~$\gamma > \gamma_+$ and~$\beta_2 > \beta_+$.
 In contrast, it follows from Theorem~\ref{th:two-stage} that there exists~$\gamma_- > 0$ such that local extinction occurs when~$\gamma < \gamma_-$ regardless of the value of the symptomatic infection rate.
 Using again monotonicity with respect to the parameter~$\gamma$ when~$\beta_1 < \beta_2$, we deduce the existence of a critical value~$\gamma_- \leq \gamma_c \leq \gamma_+$ such that
\begin{equation}
\label{eq:gammac}
\begin{array}{rcl}
\hbox{a.} \ \ \beta_1 = 0 \quad \hbox{and} \quad \gamma > \gamma_c & \Longrightarrow & \hbox{local survival for~$\beta_2 < \infty$ large}, \vspace*{4pt} \\
\hbox{b.} \ \ \beta_1 = 0 \quad \hbox{and} \quad \gamma < \gamma_c & \Longrightarrow & \hbox{local extinction even when~$\beta_2 = \infty$}.  \end{array}
\end{equation}
 This is in sharp contrast with the behavior of the mean-field model in which, provided~$\gamma > 0$, the infection survives for~$\beta_2$ sufficiently large, even when~$\beta_1 = 0$. \\ \\
{\bf Asymptomatic = supercritical and symptomatic = subcritical} \vspace*{5pt} \\
 Recall that, in this case, standard coupling arguments fail to prove monotonicity of the infected region with respect to each of the three parameters.
 However, using some obvious symmetry and exchanging the role of the two infection rates also accounting for the reverse effect of the parameter~$\gamma$, the proof of Theorem~\ref{th:top-left} easily extends to show that
\begin{theorem}
\label{th:bottom-right}
 Starting with a positive density of~1s and~2s,
\begin{equation}
\label{eq:bottom-right}
\begin{array}{rcl}
\beta_1 > \beta_c \quad \hbox{and} \quad \beta_2 < \beta_c & \Longrightarrow & \left\{\begin{array}{l}
\hbox{a. local survival for all~$\gamma > 0$ small}, \vspace*{4pt} \\
\hbox{b. local extinction for all~$\gamma < \infty$ large}. \end{array} \right. \end{array}
\end{equation}
\end{theorem}
 We now assume that~$\beta_2 = 0$, which happens for instance when all the symptomatic individuals are placed in quarantine.
 Assuming in addition that an infected individual who is asymptomatic cannot recover but turns instead into a symptomatic individual results in the three-state interacting particle system with transition rates
\begin{equation}
\label{eq:fire-rate}
\begin{array}{rclcrclcrcl}
  0 \to 1 & \hbox{at rate} & \beta_1 f_1 (x, \xi), & \quad & 1 \to 2 & \hbox{at rate} & 1 + \gamma, & \quad & 2 \to 0 & \hbox{at rate} & 1. \end{array}
\end{equation}
 This stochastic process is a time change of the standard forest fire model in which time is accelerated by the factor~$1 + \gamma$.
\begin{figure}[t!]
\centering
\scalebox{0.40}{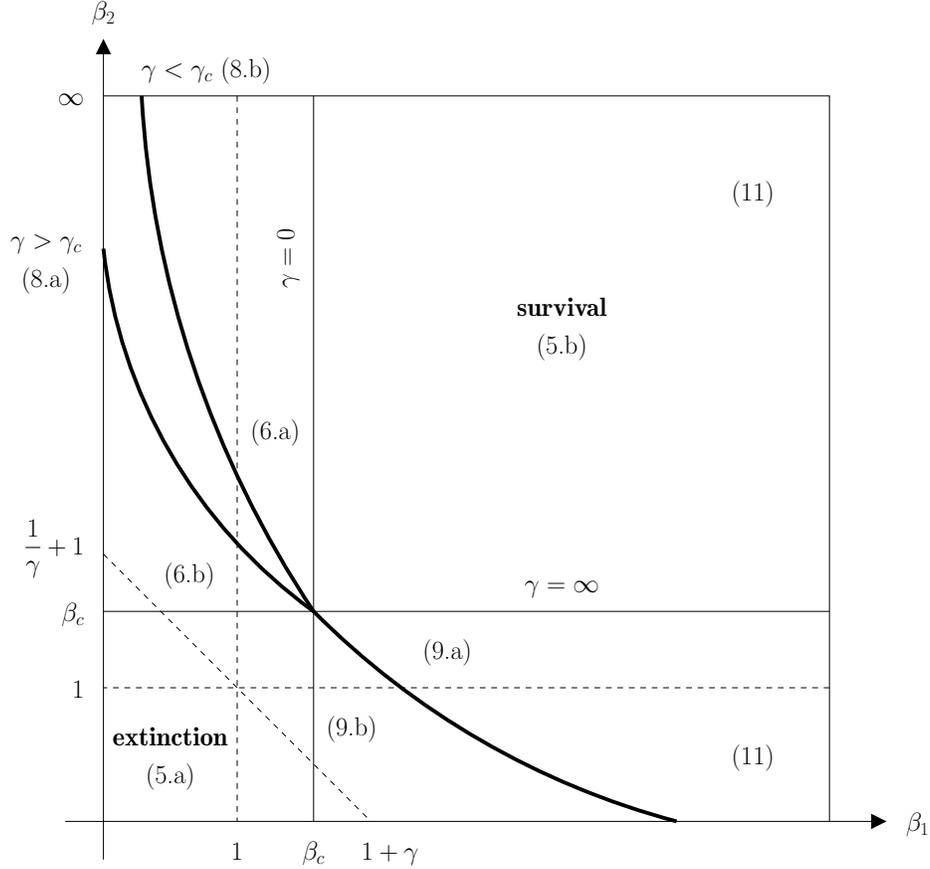}
\caption{\upshape{
 Phase structure of the contact process with an asymptomatic state and its mean-field model.
 The thick black lines show the phase transitions of the interacting particle system for various values of~$\gamma$.
 The solid vertical and horizontal lines show the phase transitions in the limiting cases~$\gamma = 0$ and~$\gamma = \infty$.
 The dashed lines are the straight lines with equation~$\beta_1 + \gamma \beta_2 = 1 + \gamma$ showing the corresponding phase transitions for the mean-field model.}}
\label{fig:PS}
\end{figure}
 In this context, the states are interpreted as
 $$ 0 = \hbox{live tree}, \qquad 1 = \hbox{burning tree}, \qquad \hbox{and} \qquad 2 = \hbox{burnt tree}. $$
 The transition rates~\eqref{eq:fire-rate} indicate that the fire spread from burning trees to nearby live trees, trees keep burning for an exponential amount of time, and burnt trees are eventually replaced by a new live tree.
 Durrett and Neuhauser~\cite{Durrett_Neuhauser_1991} proved that, in~$d = 2$, the fire keeps burning with positive probability when the rate at which burning trees emit sparks exceeds some critical value~$\alpha_c$.
 In addition, because the transition~$1 \to 0$ occurring at rate one in our process is replaced by the transition~$1 \to 2$ occurring at the same rate in the forest fire model, we expect that our epidemic model has more~1s and less~2s than the forest fire model which, accounting for the time change and using~\cite{Durrett_Neuhauser_1991}, suggests local survival when~$\beta_1 > (1 + \gamma) \,\alpha_c$.
 Though we believe that this result is true, standard coupling arguments again fail to compare both processes.
 However, using a construction inspired from Cox and Durrett~\cite{Cox_Durrett_1988}, we can compare the set of sites that ever get infected with the set of wet sites in a certain site percolation process to prove that
\begin{theorem}
\label{th:forest-fire}
 Assuming that~$d > 1$ and starting the process with a single~1 at the origin in an otherwise healthy population,
\begin{equation}
\label{eq:forest-fire}
\begin{array}{rcl} \gamma < \infty \quad \hbox{and} \quad \beta_2 = 0 & \Longrightarrow & \hbox{global survival for all} \ \beta_1 < \infty \ \hbox{large}. \end{array}
\end{equation}
\end{theorem}
 In contrast with Theorem~\ref{th:two-stage}, this result is consistent with the behavior of the mean-field model in which there is survival whenever one infection rate is sufficiently large even if the other infection rate is equal to zero.
 This shows that, while the two infection rates play similar roles in the mean-field model, they have significantly different effects in the spatial model.
 When~$\beta_1 = 0$, the~2s are quickly surrounded and blocked by the nearby~1s they create, therefore, if in addition these~1s take too much time to turn into~2s, the infection dies out.
 In contrast, when~$\beta_2 = 0$, because the~1s produce~1s, even when the~1s quickly turn into~2s, if these~1s spread fast enough they can still escape from the~2s and form traveling waves drifting off to infinity.
 Figure~\ref{fig:PS} shows a picture of the phase structure of the process and its mean-field model summarizing our results. \\
\indent
 The rest of the paper is devoted to the proofs.
 Section~\ref{sec:mean-field} focuses on the mean-field model, using in particular Bendixson--Dulac theorem to give a complete description of the global stability of the fixed points.
 Section~\ref{sec:harris} explains how to construct the process from a graphical representation and couple processes with different parameters to prove monotonicity when~$\beta_1 < \beta_2$.
 The last four sections give the proofs of the four theorems.
 Sections~\ref{sec:epidemic} and~\ref{sec:recovery} use block constructions and perturbation techniques to compare the epidemic model with~$\gamma$ small or large with a supercritical~(Section~\ref{sec:epidemic}) and a subcritical~(Section~\ref{sec:recovery}) contact process, and deduce Theorems~\ref{th:top-left} and~\ref{th:bottom-right}.
 Section~\ref{sec:two-stage} relies on a comparison of the epidemic model with a subcritical Galton--Watson branching process to prove local extinction, as stated in Theorem~\ref{th:two-stage}.
 Finally, the proof of Theorem~\ref{th:forest-fire} is carried out in Section~\ref{sec:forest-fire} by showing that the set of sites that ever get infected dominates stochastically a certain site percolation process on the integer lattice in dimension~$d > 1$.
 
%%%%%%%%%%%%%%%%%%%%%%%%%%%%%%%%%%%%%%%%%%%%%%%%%%%%%%%%%%%%%%%%%%%%%%%%%%%%%%%%%%%%%%%%%%%%%%%%%%%%%%%%%%%%%%%%%%%%%%%%%%%%%%%%%%%%%%%%%%%%%%%%%%%%%%%%%%%%%%%%%%%%%%%%%%%%%%%%%%%%%%%%%%%%%%%%%%

\section{Mean-field analysis}
\label{sec:mean-field}
 This section is devoted to the analysis of the mean-field model~(existence, local stability, and global stability of the fixed points) described in the introduction.
 Recall that, when the population is homogeneously mixing, and letting~$u_i$ be the density of individuals in state~$i$, the mean-field model consists of the system of coupled differential equations
\begin{equation}
\label{eq:mean-field}
\begin{array}{rcl}
  u_1' = F_1 (u_1, u_2) & \n = \n & (\beta_1 u_1 + \beta_2 u_2)(1 - u_1 - u_2) - (1 + \gamma) u_1, \vspace*{4pt} \\
  u_2' = F_2 (u_1, u_2) & \n = \n & \gamma u_1 - u_2. \end{array}
\end{equation}
 Setting the right-hand side of both equations equal to zero, some basic algebra shows that, in addition to the disease-free equilibrium~$p_0 = (0, 0)$ which is always a fixed point, there is also a nontrivial fixed point~$p_{12} = (\bar u_1, \bar u_2)$ whose coordinates are given by
\begin{equation}
\label{eq:fixed-point}
\bar u_1 = \frac{1}{1 + \gamma} - \frac{1}{\beta_1 + \gamma \beta_2} \quad \hbox{and} \quad \bar u_2 = \gamma \bar u_1 = \frac{\gamma}{1 + \gamma} - \frac{\gamma}{\beta_1 + \gamma \beta_2}.
\end{equation}
 This fixed point is biologically relevant in the sense that it belongs to the two-dimensional simplex if and only if~$\beta_1 + \gamma \beta_2 > 1 + \gamma$.
 To study the local stability of the two fixed points, notice first that the Jacobian matrix associated to the system~\eqref{eq:mean-field} is given by
 $$ J (u_1, u_2) = \left(\begin{array}{ccc} \beta_1 (1 - 2u_1) - (\beta_1 + \beta_2) u_2 - (1 + \gamma) && \beta_2 (1 - 2u_2) - (\beta_1 + \beta_2) u_1 \vspace*{2pt} \\ \gamma && -1 \end{array} \right). $$
 Taking~$u_1 = u_2 = 0$, the disease-free equilibrium is locally stable if and only if
 $$ \tr \,(J (0, 0)) = \beta_1 - \gamma - 2 < 0 \quad \hbox{and} \quad \de \,(J (0, 0)) = 1 + \gamma - \beta_1 - \gamma \beta_2 > 0, $$
 which is equivalent to~$\beta_1 + \gamma \beta_2 < 1 + \gamma$.
 To study the local stability of the nontrivial fixed point, we let~$D_1$ and~$D_2$ be the two denominators in the expressions~\eqref{eq:fixed-point} of the fixed points:
 $$ D_1 = 1 + \gamma \quad \hbox{and} \quad D_2 = \beta_1 + \gamma \beta_2. $$
 Multiplying the first coefficient of~$J (\bar u_1, \bar u_2)$ by~$D_1 D_2$, we get
\begin{equation}
\label{eq:first}
\begin{array}{l}
  D_1 D_2 (\beta_1 (1 - 2 \bar u_1) - (\beta_1 + \beta_2) \bar u_2 - (1 + \gamma)) \vspace*{4pt} \\ \hspace*{40pt} =
  D_1 D_2 (\beta_1 - (1 + \gamma) - ((\beta_1 + \gamma \beta_2) + (1 + \gamma) \beta_1) \bar u_1) \vspace*{4pt} \\ \hspace*{40pt} =
  D_1 D_2 (\beta_1 - D_1 - (D_2 + D_1 \beta_1)(1 / D_1 - 1 / D_2)) \vspace*{4pt} \\ \hspace*{40pt} =
  D_1 D_2 (\beta_1 - D_1) - (D_2 + D_1 \beta_1)(D_2 - D_1), \end{array}
\end{equation}
 while multiplying the second coefficient of the matrix by~$D_1 D_2$, we get
\begin{equation}
\label{eq:second}
\begin{array}{l}
  D_1 D_2 (\beta_2 (1 - 2 \bar u_2) - (\beta_1 + \beta_2) \bar u_1) \vspace*{4pt} \\ \hspace*{40pt} =
  D_1 D_2 (\beta_2 - ((\beta_1 + \gamma \beta_2) + (1 + \gamma) \beta_2) \bar u_1) \vspace*{4pt} \\ \hspace*{40pt} =
  D_1 D_2 (\beta_2 - (D_2 + D_1 \beta_2)(1 / D_1 - 1 / D_2)) \vspace*{4pt} \\ \hspace*{40pt} =
  D_1 D_2 \beta_2 - (D_2 + D_1 \beta_2)(D_2 - D_1). \end{array}
\end{equation}
 Using~\eqref{eq:first}, we deduce that the trace is negative:
 $$ \begin{array}{rcl}
 D_1 D_2  \,\tr \,(J (\bar u_1, \bar u_2)) & \n = \n & D_1 D_2 (\beta_1 - D_1) - (D_2 + D_1 \beta_1)(D_2 - D_1) - D_1 D_2 \vspace*{4pt} \\
                                           & \n = \n & - D_1^2 D_2 - D_2^2 + D_1^2 \beta_1 = - D_2^2 - D_1^2 \gamma \beta_2 < 0, \end{array} $$
 which shows that the Jacobian matrix has at least one negative eigenvalue.
 Similarly, using both expressions~\eqref{eq:first} and~\eqref{eq:second}, we deduce that the determinant is positive:
 $$ \begin{array}{l}
 D_1 D_2  \,\de \,(J (\bar u_1, \bar u_2)) \vspace*{4pt} \\ \hspace*{20pt} =
 - D_1 D_2 (\beta_1 - D_1) + (D_2 + D_1 \beta_1)(D_2 - D_1) - \gamma D_1 D_2 \beta_2 + \gamma (D_2 + D_1 \beta_2)(D_2 - D_1) \vspace*{4pt} \\ \hspace*{20pt} =
 - D_1 D_2 (\beta_1 + \gamma \beta_2 - D_1) + ((1 + \gamma) D_2 + D_1 (\beta_1 + \gamma \beta_2))(D_2 - D_1) \vspace*{4pt} \\ \hspace*{20pt} =
 - D_1 D_2 (D_2 - D_1) + (D_1 D_2 + D_1 D_2)(D_2 - D_1) =
 D_1 D_2 (D_2 - D_1) > 0, \end{array} $$
 when~$D_2 = \beta_1 + \gamma \beta_2 > 1 + \gamma = D_1$, which is exactly the condition for the nontrivial fixed point to belong to the two-dimensional simplex.
 This shows that both eigenvalues have the same sign, therefore both eigenvalues are negative and the interior fixed point~$p_{12}$ is locally stable.
\begin{figure}[t!]
\centering
\scalebox{0.40}{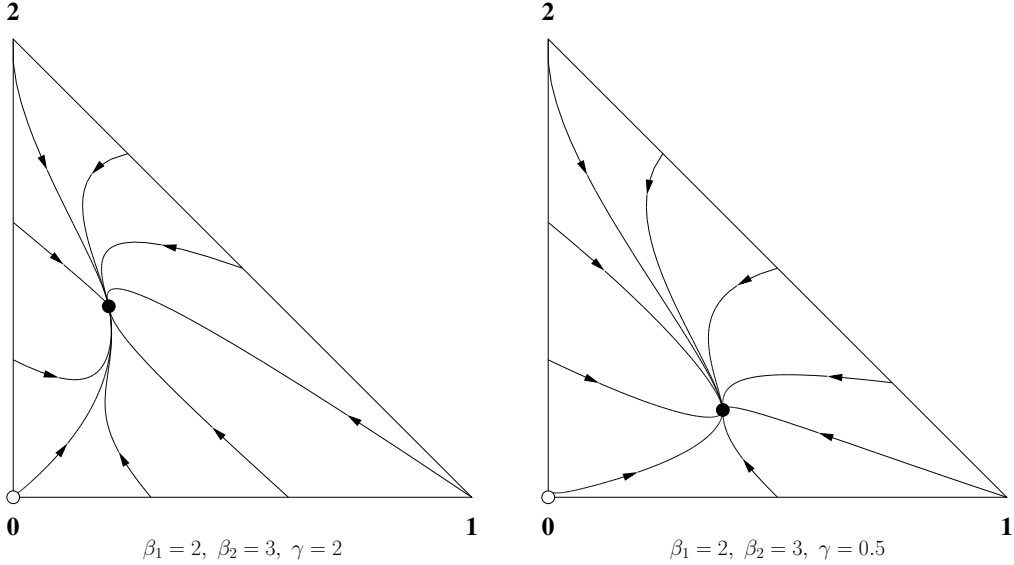}
\caption{\upshape{
 Solution curves of the mean-field model for different values of~$\gamma$.}}
\label{fig:MF}
\end{figure}
 Finally, to prove that the local stability of the fixed points under the appropriate conditions also implies global stability, the last step is to apply~Bendixson--Dulac theorem to exclude the existence of periodic orbits.
 Using the Dulac function~$\phi = 1 / (u_1 u_2)$, we obtain
 $$ \begin{array}{rcl}
    \nabla \cdot (\phi F_1, \phi F_2) & \n = \n &
    \displaystyle \frac{\partial}{\partial u_1} \bigg(\bigg(\frac{\beta_1}{u_2} + \frac{\beta_2}{u_1} \bigg) (1 - u_1 - u_2) - \frac{1 + \gamma}{u_2} \bigg) +
    \displaystyle \frac{\partial}{\partial u_2} \bigg(\frac{\gamma}{u_2} - \frac{1}{u_1} \bigg) \vspace*{8pt} \\ & \n = \n &
    \displaystyle - \bigg(\frac{\beta_2 (1 - u_1 - u_2)}{u_1^2} \bigg) - \bigg(\frac{\beta_1}{u_2} + \frac{\beta_2}{u_1} \bigg) - \frac{\gamma}{u_2^2} < 0 \end{array} $$
 for all~$(u_1, u_2)$ in the interior of the simplex.
 In conclusion,
 $$ \begin{array}{rcl} \hbox{the infection survives in the sense~\eqref{eq:mean-field-survival}} & \Longleftrightarrow & \beta_1 + \gamma \beta_2 > 1 + \gamma. \end{array} $$
 In contrast, when the inequality is flipped, the trivial fixed point~$p_0$ is the only fixed point in the simplex and, starting from any initial condition in the simplex, the system converges to~$p_0$.
 Figure~\ref{fig:MF} shows a picture of the solution curves when there is survival, while Figure~\ref{fig:PS} shows a visualization of the parameter region~$\beta_1 + \gamma \beta_2 > 1 + \gamma$~(above the oblique dashed line).
 
%%%%%%%%%%%%%%%%%%%%%%%%%%%%%%%%%%%%%%%%%%%%%%%%%%%%%%%%%%%%%%%%%%%%%%%%%%%%%%%%%%%%%%%%%%%%%%%%%%%%%%%%%%%%%%%%%%%%%%%%%%%%%%%%%%%%%%%%%%%%%%%%%%%%%%%%%%%%%%%%%%%%%%%%%%%%%%%%%%%%%%%%%%%%%%%%%%

\section{Graphical representation and monotonicity}
\label{sec:harris}
 The first step to study the spatial model is to use a key idea due to Harris~\cite{Harris_1978} that consists in constructing the process graphically from a collection of independent Poisson processes/exponential clocks.
 The concept of graphical representation not only shows that the process on the infinite lattice is well-defined~(even though the time to the first/next update does not exist) but also will be used later to couple processes with different parameters in order to prove monotonicity results, and apply block constructions/perturbation techniques to compare the process properly rescaled in space and time with oriented site percolation.
 In the case of the contact process with an asymptomatic state, the most natural approach to construct the process is to use four collections of exponential clocks corresponding to infection from an asymptomatic individual, infection from a symptomatic individual, transition from the asymptomatic state to the symptomatic state, and recovery.
\begin{figure}[t!]
\centering
\scalebox{0.35}{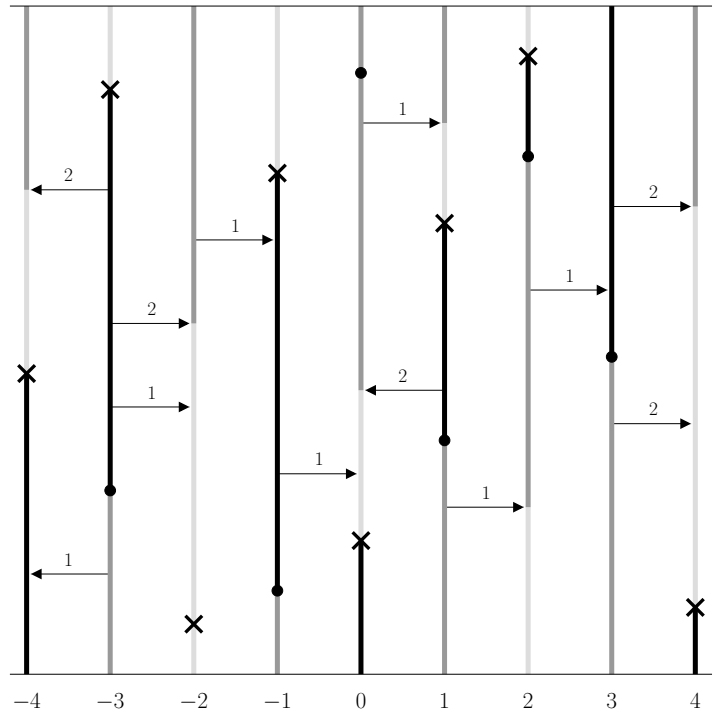}
\caption{\upshape{
 Graphical representation of the contact process with an asymptomatic state.
 The gray lines represent the infected asymptomatic individuals and the black lines the infected symptomatic individuals.}}
\label{fig:GR}
\end{figure}
 More precisely, we think of the integer lattice as the vertex set of a directed graph in which there is a directed edge~$\vec{xy}$ if and only if vertices~$x$ and~$y$ are nearest neighbors, and attach the following exponential clocks to the directed edges and vertices as follows.
\begin{itemize}
\item {\bf Infection from asymptomatic}.
           Place an exponential clock with rate~$\beta_1 / 2d$ along each directed edge~$\vec{xy}$.
           Each time the clock rings, say at time~$t$, draw an arrow~$(x, t) \overset{\hbox{\tiny 1}}{\longrightarrow} (y, t)$ to indicate that if the tail of the arrow is occupied by a~1 and the head of the arrow is healthy then the head of the arrow becomes of type~1. \vspace*{4pt}
\item {\bf Infection from symptomatic}.
           Place an exponential clock with rate~$\beta_2 / 2d$ along each directed edge~$\vec{xy}$.
           Each time the clock rings, say at time~$t$, draw an arrow~$(x, t) \overset{\hbox{\tiny 2}}{\longrightarrow} (y, t)$ to indicate that if the tail of the arrow is occupied by a~2 and the head of the arrow is healthy then the head of the arrow becomes of type~1. \vspace*{4pt}
\item {\bf Transition to symptomatic}.
           Place an exponential clock with rate~$\gamma$ at each vertex~$x$.
           Each time the clock rings, say at time~$t$, put a dot~$\bullet$ at point~$(x, t)$ to indicate that if site~$x$ is asymptomatic/occupied by a~1 then it becomes symptomatic/occupied by a~2. \vspace*{4pt}
\item {\bf Recovery}.
           Place an exponential clock with rate one at each vertex~$x$.
           Each time the clock rings, say at time~$t$, put a cross~$\times$ at point~$(x, t)$ to indicate that if site~$x$ is infected~(asymptomatic or symptomatic) then it becomes healthy. \vspace*{4pt}
\end{itemize}
 Harris~\cite{Harris_1972, Harris_1978} proved for a large class of interacting particle systems including our contact process with an asymptomatic state that the process is well-defined and can indeed be constructed using the collection of exponential clocks and algorithm above, and we refer to Figure~\ref{fig:GR} for a picture of the graphical representation and an illustration of how to construct the process. \\
\indent
 We now use different variations of the graphical representation together with the superposition property for Poisson processes to couple/construct jointly processes with different parameters and prove monotonicity with respect to each of the parameters when~$\beta_1 < \beta_2$.
 Let~$\Xi (\beta_1, \beta_2, \gamma)$ be the process with infection rate from asymptomatic individuals~$\beta_1$, infection rate from symptomatic individuals~$\beta_2$, and transition rate to the symptomatic state~$\gamma$. \\ \\
\noindent
{\bf Monotonicity in~$\beta_1$}.
 Fix~$\beta_1 < \beta_1' < \beta_2$, let
 $$ \xi_t = \Xi (\beta_1, \beta_2, \gamma) \quad \hbox{and} \quad \xi_t' = \Xi (\beta_1', \beta_2, \gamma), $$
 and construct the two processes as follows.
\begin{itemize}
 \item Place arrows~$x \longrightarrow y$ along each directed edge~$\vec{xy}$ at rate~$\beta_1 / 2d$.
       These arrows have the same effect as the~1-arrows and the~2-arrows defined above on both processes. \vspace*{4pt}
 \item Place arrows~$x \overset{\hbox{\tiny 1'}}{\longrightarrow} y$ along each directed edge~$\vec{xy}$ at rate~$(\beta_1' - \beta_1) / 2d$.
       These arrows have the same effect as the~1-arrows defined above on the process~$\xi_t'$ only and the same effect as the~2-arrows defined above on both processes. \vspace*{4pt}
 \item Place arrows~$x \overset{\hbox{\tiny 2}}{\longrightarrow} y$ along each directed edge~$\vec{xy}$ at rate~$(\beta_2 - \beta_1') /2d$.
       These arrows have the same effect as the~2-arrows defined above on both processes. \vspace*{4pt}
 \item Place black dots~$\bullet$ at each site~$x$ at rate~$\gamma$.
       These black dots have the same effect as the dots defined above on both processes. \vspace*{4pt}
 \item Place crosses~$\times$ at each site~$x$ at rate one.
       These crosses have the same effect as the crosses defined above on both processes.
\end{itemize}
 Starting both processes from the same configuration, it can be proved that all the possible pairs of states generated by the coupling~$(\xi_t, \xi_t')$ are~$S = \{00, 01, 02, 11, 12, 22 \}$, i.e., the set~$S$ is closed under the dynamics.
\begin{table}[h!]
\begin{tabular}{c|p{1em}p{0.4em}p{0.4em}p{0.4em}p{0.4em}p{0.4em}p{1em}|p{1em}p{0.4em}p{0.4em}p{0.4em}p{0.4em}p{0.4em}p{1em}|p{1em}p{0.4em}p{0.4em}p{0.4em}p{0.4em}p{0.4em}p{1em}|c|c}
    & $\longrightarrow$ & 00 & 01 & 02 & 11 & 12 & 22 & $\overset{\hbox{\tiny 1'}}{\longrightarrow}$ & 00 & 01 & 02 & 11 & 12 & 22 & $\overset{\hbox{\tiny 2}}{\longrightarrow}$ & 00 & 01 & 02 & 11 & 12 & 22 & $\bullet$ & $\times $ \\ \hline
 00 && 00 & 01 & 02 & 11 & 12 & 22 && 00 & 01 & 02 & 11 & 12 & 22 && 00 & 01 & 02 & 11 & 12 & 22 & 00 & 00 \\
 01 && 01 & 01 & 02 & 11 & 12 & 22 && 01 & 01 & 02 & 11 & 12 & 22 && 00 & 01 & 02 & 11 & 12 & 22 & 02 & 00 \\
 02 && 01 & 01 & 02 & 11 & 12 & 22 && 01 & 01 & 02 & 11 & 12 & 22 && 01 & 01 & 02 & 11 & 12 & 22 & 02 & 00 \\
 11 && 11 & 11 & 12 & 11 & 12 & 22 && 01 & 01 & 02 & 11 & 12 & 22 && 00 & 01 & 02 & 11 & 12 & 22 & 22 & 00 \\
 12 && 11 & 11 & 12 & 11 & 12 & 22 && 01 & 01 & 02 & 11 & 12 & 22 && 01 & 01 & 02 & 11 & 12 & 22 & 22 & 00 \\
 22 && 11 & 11 & 12 & 11 & 12 & 22 && 11 & 11 & 12 & 11 & 12 & 22 && 11 & 11 & 12 & 11 & 12 & 22 & 22 & 00
\end{tabular} \vspace*{10pt}
\caption{\upshape{Coupling showing the monotonicity with respect to~$\beta_1$ when~$\beta_1 < \beta_2$}}
\label{tab:beta1}
\end{table}
 See~Table~\ref{tab:beta1} where the pairs of states in the first column are the possible pairs of states at site~$x$ before a potential update, the pairs of states in the first row are the possible pairs of states at~$y$ before a potential update, and the pairs of states in the table are the pairs of states resulting from the dynamics.
 This implies that the process~$\xi_t'$ has more infected sites than the process~$\xi_t$ therefore the set of infected sites is nondecreasing with respect to~$\beta_1$. \\ \\
\noindent
{\bf Monotonicity in~$\beta_2$}.
 Fix~$\beta_1 < \beta_2 < \beta_2'$, let
 $$ \xi_t = \Xi (\beta_1, \beta_2, \gamma) \quad \hbox{and} \quad \xi_t' = \Xi (\beta_1, \beta_2', \gamma), $$
 and construct the two processes as follows.
\begin{itemize}
 \item Place arrows~$x \longrightarrow y$ along each directed edge~$\vec{xy}$ at rate~$\beta_1 / 2d$.
       These arrows have the same effect as the~1-arrows and the~2-arrows defined above on both processes. \vspace*{4pt}
 \item Place arrows~$x \overset{\hbox{\tiny 2}}{\longrightarrow} y$ along each directed edge~$\vec{xy}$ at rate~$(\beta_2 - \beta_1) / 2d$.
       These arrows have the same effect as the~2-arrows defined above on both processes. \vspace*{4pt}
 \item Place arrows~$x \overset{\hbox{\tiny 2'}}{\longrightarrow} y$ along each directed edge~$\vec{xy}$ at rate~$(\beta_2' - \beta_2) / 2d$.
       These arrows have the same effect as the~2-arrows defined above on the process~$\xi_t'$ only. \vspace*{4pt}
 \item Place black dots~$\bullet$ at each site~$x$ at rate~$\gamma$.
       These black dots have the same effect as the dots defined above on both processes. \vspace*{4pt}
 \item Place crosses~$\times$ at each site~$x$ at rate one.
       These crosses have the same effect as the crosses defined above on both processes.
\end{itemize}
\begin{table}[h!]
\begin{tabular}{c|p{1em}p{0.4em}p{0.4em}p{0.4em}p{0.4em}p{0.4em}p{1em}|p{1em}p{0.4em}p{0.4em}p{0.4em}p{0.4em}p{0.4em}p{1em}|p{1em}p{0.4em}p{0.4em}p{0.4em}p{0.4em}p{0.4em}p{1em}|c|c}
    & $\longrightarrow$ & 00 & 01 & 02 & 11 & 12 & 22 & $\overset{\hbox{\tiny 2}}{\longrightarrow}$ & 00 & 01 & 02 & 11 & 12 & 22 & $\overset{\hbox{\tiny 2'}}{\longrightarrow}$ & 00 & 01 & 02 & 11 & 12 & 22 & $\bullet$ & $\times $ \\ \hline
 00 && 00 & 01 & 02 & 11 & 12 & 22 && 00 & 01 & 02 & 11 & 12 & 22 && 00 & 01 & 02 & 11 & 12 & 22 & 00 & 00 \\
 01 && 01 & 01 & 02 & 11 & 12 & 22 && 01 & 01 & 02 & 11 & 12 & 22 && 00 & 01 & 02 & 11 & 12 & 22 & 02 & 00 \\
 02 && 01 & 01 & 02 & 11 & 12 & 22 && 01 & 01 & 02 & 11 & 12 & 22 && 01 & 01 & 02 & 11 & 12 & 22 & 02 & 00 \\
 11 && 11 & 11 & 12 & 11 & 12 & 22 && 01 & 01 & 02 & 11 & 12 & 22 && 00 & 01 & 02 & 11 & 12 & 22 & 22 & 00 \\
 12 && 11 & 11 & 12 & 11 & 12 & 22 && 01 & 01 & 02 & 11 & 12 & 22 && 01 & 01 & 02 & 11 & 12 & 22 & 22 & 00 \\
 22 && 11 & 11 & 12 & 11 & 12 & 22 && 11 & 11 & 12 & 11 & 12 & 22 && 11 & 11 & 12 & 11 & 12 & 22 & 22 & 00
\end{tabular} \vspace*{10pt}
\caption{\upshape{Coupling showing the monotonicity with respect to~$\beta_2$ when~$\beta_1 < \beta_2$}}
\label{tab:beta2}
\end{table}
 As previously, Table~\ref{tab:beta2} shows that the set~$S$ is closed under the dynamics of the coupling~$(\xi_t, \xi_t')$, which implies monotonicity with respect to the infection rate~$\beta_2$. \\ \\
\noindent
{\bf Monotonicity in~$\gamma$}.
 Fix~$\beta_1 < \beta_2$ and~$\gamma < \gamma'$, let
 $$ \xi_t = \Xi (\beta_1, \beta_2, \gamma) \quad \hbox{and} \quad \xi_t' = \Xi (\beta_1, \beta_2, \gamma'), $$
 and construct the two processes as follows.
\begin{itemize}
 \item Place arrows~$x \longrightarrow y$ along each directed edge~$\vec{xy}$ at rate~$\beta_1 / 2d$.
       These arrows have the same effect as the~1-arrows and the~2-arrows defined above on both processes. \vspace*{4pt}
 \item Place arrows~$x \overset{\hbox{\tiny 2}}{\longrightarrow} y$ along each directed edge~$\vec{xy}$ at rate~$(\beta_2 - \beta_1) / 2d$.
       These arrows have the same effect as the~2-arrows defined above on both processes. \vspace*{4pt}
 \item Place black dots~$\bullet$ at each site~$x$ at rate~$\gamma$.
       These black dots have the same effect as the dots defined above on both processes. \vspace*{4pt}
 \item Place white dots~$\circ$ at each site~$x$ at rate~$\gamma' - \gamma$.
       These white dots have the same effect as the dots defined above on the process~$\xi_t'$ only. \vspace*{4pt}
 \item Place crosses~$\times$ at each site~$x$ at rate one.
       These crosses have the same effect as the crosses defined above on both processes.
\end{itemize} 
\begin{table}[h!]
\begin{tabular}{c|p{1em}p{0.4em}p{0.4em}p{0.4em}p{0.4em}p{0.4em}p{1em}|p{1em}p{0.4em}p{0.4em}p{0.4em}p{0.4em}p{0.4em}p{1em}|c|c|c}
    & $\longrightarrow$ & 00 & 01 & 02 & 11 & 12 & 22 & $\overset{\hbox{\tiny 2}}{\longrightarrow}$ & 00 & 01 & 02 & 11 & 12 & 22 & $\bullet$ & $\circ$ & $\times$ \\ \hline
 00 && 00 & 01 & 02 & 11 & 12 & 22 && 00 & 01 & 02 & 11 & 12 & 22 & 00 & 00 & 00 \\
 01 && 01 & 01 & 02 & 11 & 12 & 22 && 00 & 01 & 02 & 11 & 12 & 22 & 02 & 02 & 00 \\
 02 && 01 & 01 & 02 & 11 & 12 & 22 && 01 & 01 & 02 & 11 & 12 & 22 & 02 & 02 & 00 \\
 11 && 11 & 11 & 12 & 11 & 12 & 22 && 00 & 01 & 02 & 11 & 12 & 22 & 22 & 12 & 00 \\
 12 && 11 & 11 & 12 & 11 & 12 & 22 && 01 & 01 & 02 & 11 & 12 & 22 & 22 & 12 & 00 \\
 22 && 11 & 11 & 12 & 11 & 12 & 22 && 11 & 11 & 12 & 11 & 12 & 22 & 22 & 22 & 00 \\
\end{tabular} \vspace*{10pt}
\caption{\upshape{Coupling showing the monotonicity with respect to~$\gamma$ when~$\beta_1 < \beta_2$}}
\label{tab:gamma}
\end{table}
 Table~\ref{tab:gamma} confirms that the set~$S$ is again closed under the dynamics of the coupling~$(\xi_t, \xi_t')$ therefore the set of infected sites is nondecreasing with respect to~$\gamma$. \\
\indent
 Using similar joint graphical representations in the case~$\beta_1 > \beta_2$ and starting both processes from the same configuration first create the same pairs of states as before, including~12.
 However, because the asymptomatic individuals are more infectious, the graphical representation now contains type~1 arrows through which only the~1s can infect nearby healthy individuals.
 In particular, the presence of a type~1 arrow going from a site in state~12 to a site in state~00 creates a site in state~10, thus showing that a standard coupling fails to prove monotonicity when~$\beta_1 > \beta_2$.

%%%%%%%%%%%%%%%%%%%%%%%%%%%%%%%%%%%%%%%%%%%%%%%%%%%%%%%%%%%%%%%%%%%%%%%%%%%%%%%%%%%%%%%%%%%%%%%%%%%%%%%%%%%%%%%%%%%%%%%%%%%%%%%%%%%%%%%%%%%%%%%%%%%%%%%%%%%%%%%%%%%%%%%%%%%%%%%%%%%%%%%%%%%%%%%%%%

\section{Proof of Theorems~\ref{th:top-left}.a and~\ref{th:bottom-right}.a (survival)}
\label{sec:epidemic}
 This section is devoted to the survival parts of Theorems~\ref{th:top-left} and~\ref{th:bottom-right}.
 The intuition behind these results is simple.
 When~$\gamma = 0$, the epidemic model reduces to a contact process with parameter~$\beta_1$ so the infection survives locally when~$\beta_1 > \beta_c$.
 Similarly, when~$\gamma = \infty$, the epidemic model reduces to a contact process with parameter~$\beta_2$ so the infection survives locally when~$\beta_2 > \beta_c$.
 To extend these results to~$\gamma$ small/large, we apply a standard perturbation argument to the block construction introduced by Bezuidenhout and Grimmett~\cite{Bezuidenhout_Grimmett_1990} to study the supercritical contact process. \\ \\
{\bf Proof of Theorem~\ref{th:bottom-right}.a}.
 Let~$\beta_1 > \beta_c$ and think of the contact process
 $$ \zeta_t^0 : \Z^d \longrightarrow \{0, 1 \} \quad \hbox{where} \quad 0 = \hbox{healthy} \quad \hbox{and} \quad 1 = \hbox{infected} $$
 with parameter~$\beta_1$ as being constructed from the graphical representation depicted in Figure~\ref{fig:GR} by using the type~1 infection arrows and the recovery crosses~(but ignoring the type~2 arrows and the dots).
 Note that, denoting by~$P_{\gamma}$ the probability associated to the epidemic model with parameter~$\gamma$, the probability associated to the contact process is simply~$P_0$.
 Now let
 $$ \Lat_1 = \{(m, n) \in \Z \times \Z_+ : m + n \ \hbox{is even} \}, $$
 which we turn into a directed graph~$\vec \Lat_1$ by placing arrows
 $$ \begin{array}{rcl} (m, n) \to (m', n') & \Longleftrightarrow & |m - m'| = 1 \ \ \hbox{and} \ \ n' = n + 1. \end{array} $$
 The idea of the block construction is to compare the contact process properly rescaled in space and time with oriented site percolation on this directed graph.
 Having four integers~$J, K, L, T$ playing the role of space and time scales, we let~$D = [\![-J, J ]\!]^d \times \{0 \}$ and
 $$ \begin{array}{rcl}
      A_{m, n} & \n = \n & (2mL e_1, n KT) + [-2L, 2L]^d \times [0, (K + 1) T] \vspace*{4pt} \\
      B_{m, n} & \n = \n & (2mL e_1, (n + 1) KT) + [-2L, 2L]^d \times [0, T] \end{array} $$
 for all~$(m, n) \in \Lat_1$ where~$e_1$ is the first unit vector in~$\Z^d$.
 Let also
 $$ E_{m, n} = \hbox{``there is a translate of~$D$ contained in~$B_{m, n}$ that is fully infected''}, $$
 and declare site~$(m, n)$ to be a good site when~$E_{m, n}$ occurs.
 Bezuidenhout and Grimmett proved that, for all infection rates~$\beta_1 > \beta_c$ and all~$\ep > 0$ small, the scale parameters can be chosen large enough that the set of good sites dominates stochastically the set of wet sites in an oriented site percolation process on~$\vec \Lat_1$ with parameter~$1 - \ep$.
 More precisely, there exists a collection of good events~$G_{m, n}$ that only depend on the graphical representation in~$A_{m, n}$ such that
\begin{equation}
\label{eq:epidemic-1}
  P_0 (G_{m, n}) \geq 1 - \ep \quad \hbox{and} \quad E_{m, n} \cap G_{m, n} \subset E_{m - 1, n + 1} \cap E_{m + 1, n + 1}.
\end{equation}
 In other words, a translate of~$D$ fully infected in~$B_{m, n}$ will produce, with probability at least~$1 - \ep$, two translates of~$D$ fully infected in the two blocks immediately above.
 Choosing~$\ep > 0$ small so that the percolation process with parameter~$1 - \ep$ is supercritical, this construction shows that the contact process with parameter~$\beta_1 > \beta_c$ survives locally.
 Returning to the epidemic model, we denote by~$G_{m, n}^0$ the event that there are no dots in the block~$A_{m, n}$.
 On this event, the epidemic model and the contact process share the same graphical representation in the block so
\begin{equation}
\label{eq:epidemic-2}
  E_{m, n} \cap (G_{m, n} \cap G_{m, n}^0) \subset E_{m - 1, n + 1} \cap E_{m + 1, n + 1}.
\end{equation}
 In addition, the scale parameters being fixed, letting~$S = (4L + 1)^d (K + 1) T$ be the size~(number of sites times length of the time interval) of~$A_{m, n}$, there exists~$\gamma_- > 0$ small such that
\begin{equation}
\label{eq:epidemic-3}
  P_{\gamma} (G_{m, n}^0) = P (\poisson (\gamma S) = 0) = \exp (- \gamma S) \geq 1 - \ep
\end{equation}
 for all~$\gamma < \gamma_-$.
 In particular, for all~$\gamma < \gamma_-$,
\begin{equation}
\label{eq:epidemic-4}
  P_{\gamma} (G_{m, n} \cap G_{m, n}^0) \geq 1 - P_0 (G_{m, n}^c) - P_{\gamma} ((G_{m, n}^0)^c) \geq 1 - 2 \ep.
\end{equation}
 Combining~\eqref{eq:epidemic-2} and~\eqref{eq:epidemic-4} shows that~\eqref{eq:epidemic-1} also holds~(with~$2 \ep$ instead of~$\ep$ and different good events) for the epidemic model with~$\gamma < \gamma_-$.
 Because~$\ep$ can be chosen arbitrarily small, this implies that the infection survives locally, which proves Theorem~\ref{th:bottom-right}.a. \\ \\
{\bf Proof of Theorem~\ref{th:top-left}.a}.
 Let~$\beta_2 > \beta_c$ and think of the contact process
 $$ \zeta_t^{\infty} : \Z^d \longrightarrow \{0, 2 \} \quad \hbox{where} \quad 0 = \hbox{healthy} \quad \hbox{and} \quad 2 = \hbox{infected} $$
 with parameter~$\beta_2$ as being constructed from the graphical representation depicted in Figure~\ref{fig:GR} by only using the type~2 infection arrows and the recovery crosses.
 Note that the probability associated to this contact process is now~$P_{\infty}$.
 In addition, because the process is supercritical, the two statements in~\eqref{eq:epidemic-1} still hold but for the probability~$P_{\infty}$ instead of~$P_0$, i.e.,
\begin{equation}
\label{eq:epidemic-5}
  P_{\infty} (G_{m, n}) \geq 1 - \ep \quad \hbox{and} \quad E_{m, n} \cap G_{m, n} \subset E_{m - 1, n + 1} \cap E_{m + 1, n + 1}.
\end{equation}
 To prove local survival of the infection when~$\gamma$ is large, note that the process behaves like its coupled contact process if each time a~2 in the contact process tries to infect a neighbor, the same happens in the epidemic model, and if none of the 1s in the epidemic model tries to infect a neighbor.
 From the point of view of the graphical representation, this occurs in the block~$A_{m, n}$ if each time there is an arrow pointing at~$(x, t)$ in the block, there is a dot at site~$x$ before the next arrow starting at site~$x$.
 In particular, denoting this event by~$G_{m, n}^{\infty}$, the analog of~\eqref{eq:epidemic-2} holds:
\begin{equation}
\label{eq:epidemic-6}
  E_{m, n} \cap (G_{m, n} \cap G_{m, n}^{\infty}) \subset E_{m - 1, n + 1} \cap E_{m + 1, n + 1}.
\end{equation}
 Following the same approach as before, the last step to prove the theorem is to show that the probability of~$G_{m, n}^{\infty}$ is close to one when~$\gamma$ large.
 To do this, let
 $$ H_{m, n} = \hbox{``there are less than~$2H = 2 (\beta_1 + \beta_2) S$ arrows pointing at~$A_{m, n}$''}, $$
 where~$S$ is the size of~$A_{m, n}$.
 Because the arrows occur at rate~$\beta_1 + \beta_2$ at each site,
\begin{equation}
\label{eq:epidemic-7}
  P_{\gamma} (H_{m, n}^c) = P (\poisson ((\beta_1 + \beta_2) S) \geq 2 (\beta_1 + \beta_2) S) \leq \ep / 2
\end{equation}
 for all~$K, L, T$ large.
 Now, let~$X$ and~$Y$ be exponentially distributed with rate~$\beta_1 + \beta_2$ and~$\gamma$, respectively.
 The scale parameters being fixed so that~\eqref{eq:epidemic-5} and~\eqref{eq:epidemic-7} hold, because the dots appear at rate~$\gamma$ at each site, it follows from the superposition property that
\begin{equation}
\label{eq:epidemic-8}
  P_{\gamma} ((G_{m, n}^{\infty})^c \,| \,H_{m, n}) \leq 2H P (X < Y) = 2H \bigg(\frac{\beta_1 + \beta_2}{\beta_1 + \beta_2 + \gamma} \bigg) \leq \ep / 2
\end{equation}
 for all~$\gamma$ larger than some~$\gamma_+ < \infty$.
 Combining~\eqref{eq:epidemic-7} and~\eqref{eq:epidemic-8}, we deduce that
 $$ \begin{array}{rcl}
      P_{\gamma} ((G_{m, n}^{\infty})^c) & \n = \n &
      P_{\gamma} ((G_{m, n}^{\infty})^c \,| \,H_{m, n}) P_{\gamma} (H_{m, n}) + P_{\gamma} ((G_{m, n}^{\infty})^c \,| \,H_{m, n}^c) P_{\gamma} (H_{m, n}^c) \vspace*{4pt} \\ & \n \leq \n &
      P_{\gamma} ((G_{m, n}^{\infty})^c \,| \,H_{m, n}) + P_{\gamma} (H_{m, n}^c) \leq \ep, \end{array} $$
 which, together with~\eqref{eq:epidemic-5}, implies that, for all~$\gamma > \gamma_+$,
\begin{equation}
\label{eq:epidemic-9}
  P_{\gamma} (G_{m, n} \cap G_{m, n}^{\infty}) \geq 1 - P_{\infty} (G_{m, n}^c) - P_{\gamma} ((G_{m, n}^{\infty})^c) \geq 1 - 2 \ep.
\end{equation}
 As previously, combining~\eqref{eq:epidemic-6} and~\eqref{eq:epidemic-9} implies that, for all~$\gamma > \gamma_+$, the infection survives locally, which completes the proof of Theorem~\ref{th:top-left}.a.

%%%%%%%%%%%%%%%%%%%%%%%%%%%%%%%%%%%%%%%%%%%%%%%%%%%%%%%%%%%%%%%%%%%%%%%%%%%%%%%%%%%%%%%%%%%%%%%%%%%%%%%%%%%%%%%%%%%%%%%%%%%%%%%%%%%%%%%%%%%%%%%%%%%%%%%%%%%%%%%%%%%%%%%%%%%%%%%%%%%%%%%%%%%%%%%%%%

\section{Proof of Theorems~\ref{th:top-left}.b and~\ref{th:bottom-right}.b (extinction)}
\label{sec:recovery}
 This section is devoted to the extinction parts of Theorems~\ref{th:top-left} and~\ref{th:bottom-right}.
 The proof relies on the same three ingredients as for the survival parts:
 coupling of the epidemic model with the contact processes~$\zeta_t^0$ and~$\zeta_t^{\infty}$, block construction, and perturbation arguments at~$\gamma = 0$ and at~$\gamma = \infty$.
 The main difference with the proof of survival is the block construction. \\ \\
{\bf Proof of Theorem~\ref{th:top-left}.b}.
 As previously, let~$\zeta_t^0$ be the basic contact process with parameter~$\beta_1$ coupled with the epidemic model, and assume that~$\beta_1 < \beta_c$.
 To compare the process properly rescaled in space and time with oriented site percolation, we now let~$\Lat_2 = \Z^d \times \Z_+$, which we turn into a directed graph~$\vec \Lat_2$ by placing arrows
 $$ \begin{array}{rcl}
     (m, n) \to (m', n') & \Longleftrightarrow & |m_1 - m_1'| + \cdots + |m_d - m_d'| + |n - n'|= 1 \ \ \hbox{and} \ \ n \leq n'. \end{array} $$
 One can think of the horizontal arrows as expansion in space and of the vertical arrows as persistence in time.
 Having an integer~$L$ playing the role of the space and time scales, we let
 $$ \begin{array}{rcl}
      A_{m, n} & \n = \n & (2mL, nL) + [-2L, 2L]^d \times [0, 2L] \vspace*{4pt} \\
      B_{m, n} & \n = \n & (2mL, nL) + [-L, L]^d \times [L, 2L] \end{array} $$
 for all~$(m, n) \in \Lat_2$.
 Then, we define the events
 $$ E_{m, n} = \hbox{``the space-time block~$B_{m, n}$ is fully healthy''}, $$
 and declare site~$(m, n)$ to be a good site when~$E_{m, n}$ occurs.
 To prove that the set of good sites dominates the set of wet sites in an oriented site percolation process on~$\vec \Lat_2$ with parameter arbitrarily close to one, the idea is to show that, with high probability when~$L$ is large, none of the infection paths starting from the bottom or the periphery of~$A_{m, n}$ reaches the block~$B_{m, n}$.
 For~$s < t$, we say that there is an infection path from~$(y, s)$ to~$(x, t)$ if we can go from the first space-time point to the second one going up in the graphical representation, crossing the infection arrows from tail to head but avoiding the recovery crosses.
 Bezuidenhout and Grimmett~\cite{Bezuidenhout_Grimmett_1991} proved the following exponential decay for the radius and duration of the infection paths in the subcritical phase:
 starting with a single infection at the origin, and identifying~$\zeta_s^0$ with the set of infected sites at time~$s$,
\begin{equation}
\label{eq:extinction-1}
  P \bigg(\bigcup_{s \geq 0} \zeta_s^0 \not \subset [-r, r]^d \bigg) \leq e^{- \alpha_1 r} \quad \hbox{and} \quad P (\zeta_t^0 \neq \varnothing) \leq e^{- \alpha_2 t}
\end{equation}
 for some~$\alpha_1, \alpha_2 > 0$.
 Let~$\ep > 0$ be small, and let
 $$ S_p = 2L ((4L + 1)^d - (4L - 1)^d) \quad \hbox{and} \quad S_b = (4L + 1)^2 $$
 be the size~(number of sites times length of the time interval) of the periphery of the block~$A_{m, n}$, and the size~(number of sites) of the bottom of the block.
 Consider also the events
 $$ H_{m, n} = \hbox{``there are less than~$2 \beta_1 S_p$ arrows starting from the pheriphery of~$A_{m, n}$''}. $$
 Using that an infection path starting from the periphery or bottom of the block~$A_{m, n}$ must have spatial/temporal length at least~$L$ to reach~$B_{m, n}$, it follows from~\eqref{eq:extinction-1} that
\begin{equation}
\label{eq:extinction-2}
\begin{array}{l}
  P (\hbox{path from periphery or bottom of~$A_{m, n}$ reaching~$B_{m, n}$}) \vspace*{4pt} \\ \hspace*{30pt} \leq
  P (\hbox{path from periphery} \,| \,H_{m, n}) P (H_{m, n}) \vspace*{4pt} \\ \hspace*{60pt} + \
  P (\hbox{path from periphery} \,| \,H_{m, n}^c) P (H_{m, n}^c) + P (\hbox{path from bottom}) \vspace*{4pt} \\ \hspace*{30pt} \leq
  P (\hbox{path from periphery} \,| \,H_{m, n}) + P (H_{m, n}^c) + P (\hbox{path from bottom}) \vspace*{4pt} \\ \hspace*{30pt} \leq
  2 \beta_1 S_p \,e^{- \alpha_1 L} + P (\poisson (\beta_1 S_p) \geq 2 \beta_1 S_p) + S_b \,e^{- \alpha_2 L} \leq \ep \end{array}
\end{equation}
 for all~$L$ sufficiently large.
 In particular, there exists a collection of good events~$G_{m, n}$ that only depend on the graphical representation in~$A_{m, n}$ such that
\begin{equation}
\label{eq:extinction-3}
  P_0 (G_{m, n}) \geq 1 - \ep \quad \hbox{and} \quad G_{m, n} \subset E_{m, n}.
\end{equation}
 In words, for all infection rates~$\beta_1 < \beta_c$ and all~$\ep > 0$ small, the scale parameter~$L$ can be chosen large enough that the set of good sites dominates stochastically the set of wet sites in an oriented site percolation process on~$\vec \Lat_2$ with parameter~$1 - \ep$.
 For~$\ep > 0$ small enough, not only the set of open sites percolates but also the probability of a path of closed~(not open) sites with length at least~$n$ starting at~$(0, 0)$ decays exponentially with~$n$~(see~\cite[Section~8]{vandenBerg_Grimmett_Schinazi_1998} for a proof).
 Because the infection cannot appear spontaneously, this shows local extinction.
 Using a perturbation argument similar to~\eqref{eq:epidemic-3} extends this result to the epidemic model with~$\gamma > 0$ small. \\ \\
{\bf Proof of Theorem~\ref{th:bottom-right}.b}.
 As in the previous section, let~$\zeta_t^{\infty}$ be the~(subcritical) contact process with parameter~$\beta_2 < \beta_c$ coupled with the epidemic model.
 Then, the two properties in~\eqref{eq:extinction-3} again hold for this process and can be extended to the epidemic model with~$\gamma < \infty$ large~(depending on the scale parameter~$L$) using the estimates~\eqref{eq:epidemic-7} and~\eqref{eq:epidemic-8}.

%%%%%%%%%%%%%%%%%%%%%%%%%%%%%%%%%%%%%%%%%%%%%%%%%%%%%%%%%%%%%%%%%%%%%%%%%%%%%%%%%%%%%%%%%%%%%%%%%%%%%%%%%%%%%%%%%%%%%%%%%%%%%%%%%%%%%%%%%%%%%%%%%%%%%%%%%%%%%%%%%%%%%%%%%%%%%%%%%%%%%%%%%%%%%%%%%%

\section{Proof of Theorem~\ref{th:two-stage}}
\label{sec:two-stage}
 Throughout this section, we assume that~$\beta_1 = 0$, meaning that the~1s are not contagious.
 Because only the~2s can infect nearby healthy individuals, in order to prove local extinction of the infection when~$\gamma$ is small, the idea is to show that the set of~2s is dominated by a subcritical Galton--Watson branching process.
 More precisely, the first objective is to prove that, starting with a single~2 in an otherwise healthy population, the infection paths that originate from this~2 cannot expand too far or live too long.
 The next lemma shows how to control the radius of the infection.
\begin{lemma}
\label{lem:GWBP-radius}
 Let~$\beta_1 = 0$ and~$\gamma < 1 / (4d - 1)$.
 For all~$r > 0$, starting with a single~2 at the origin in an otherwise healthy population, and identifying~$\xi_s$ with the set of infected sites,
 $$ P \bigg(\bigcup_{s \geq 0} \xi_s \not \subset [-r, r]^d \bigg) \leq \bigg(\frac{4d \gamma}{1 + \gamma} \bigg)^{r - 1} \to 0 \quad \hbox{as} \ r \to \infty. $$
\end{lemma}
\begin{proof}
 In view of the monotonicity of the set of infected sites with respect to the symptomatic infection rate when~$\beta_1 < \beta_2$, it suffices to prove the result when~$\beta_2 = \infty$.
 In this case, the process is dominated stochastically by a subcritical Galton--Watson branching process.
 Each time a~2 at site~$x$ gives birth to a~1 at a neighboring site~$y \in N_x$ that turns into a~2,
\begin{itemize}
 \item we think of the~2 at site~$x$ as the parent of the~2 at site~$y$ and \vspace*{4pt}
 \item we think of the~2 at site~$y$ as the child of the~2 at site~$x$,
\end{itemize}
 which induces a partition of the set of~2s into generations, starting with the initial single~2 belonging to generation zero.
 The domination of this system by a subcritical Galton--Watson branching process follows from an argument due to Foxall~\cite[Proposition~2]{Foxall_2015} that we briefly explain.
 In the limiting case~$\beta_2 = \infty$, each~2 gives birth instantaneously to a~1 onto each of its healthy neighbors, whose number is bounded by the degree~$2d$ of the underlying graph.
 An additional~1 appears instantaneously for each recovery mark in the neighborhood of~$x$ occurring before the next recovery mark at~$x$.
 By the superposition property, the number of such recovery marks is the sum of~$2d$ independent shifted geometric random variables with parameter one-half.
 In particular, the expected number of~1s generated by the~2 at site~$x$ is bounded by
\begin{equation}
\label{eq:GWBP-radius-1}
  2d + 2d \,E (\geometric (1/2) - 1) = 2d + 2d (2 - 1) = 4d.
\end{equation}
 Using again the superposition property implies that each of these~1s turns into a~2 before it recovers with probability~$\gamma / (1 + \gamma)$, so the number of~2s at generation~$n$ is dominated by a Galton--Watson branching process~$X_n$ in which the mean number of offspring per individual is
 $$ \mu = 4d \gamma / (1 + \gamma) < 1. $$
 Using that the process~$\mu^{-n} X_n$ is a martingale, the probability that this Galton--Watson branching process lives for at least~$n$ generations is bounded by
\begin{equation}
\label{eq:GWBP-radius-2}
  P (X_n \neq 0) \leq E (X_n) = \mu^n E (X_0) = \mu^n = (4d \gamma / (1 + \gamma))^n.
\end{equation}
 To prove that the infection generated by the initial~2 cannot spread too far, note that, because it takes at least~$r - 1$ generations to bring a~2 outside the spatial box~$[-r + 1, r - 1]^d$, which is necessary to produce an infected site outside the box~$[-r, r]^d$,
 $$ P \bigg(\bigcup_{s \geq 0} \xi_s \not \subset [-r, r]^d \bigg) \leq P (X_{r - 1} \neq 0) \leq \bigg(\frac{4d \gamma}{1 + \gamma} \bigg)^{r - 1} $$
 according to~\eqref{eq:GWBP-radius-2}.
 This completes the proof.
\end{proof} \\ \\
 Using the previous branching process, we now control the duration of the infection.
\begin{lemma}
\label{lem:GWBP-time}
 Let~$\beta_1 = 0$ and~$\gamma < 1 / (4d - 1)$.
 Then, starting with a single~2 at the origin in an otherwise healthy population,
 $$ P (T_0 < \infty) = 1 \quad \hbox{where} \quad T_0 = \inf \{t : \xi_t = \varnothing \}. $$
\end{lemma}
\begin{proof}
 Let~$N_2$ be the total number of individuals in the Galton--Watson branching process introduced in the previous proof, and let~$N_1$ be the total number of~1s generated by these~2s.
 By the monotone convergence theorem and the exponential decay in~\eqref{eq:GWBP-radius-2},
\begin{equation}
\label{eq:GWBP-time-1}
  E (N_2) = E \bigg(\sum_{n = 0}^{\infty} \,X_n \bigg) = \sum_{n = 0}^{\infty} \,E (X_n) = \sum_{n = 0}^{\infty} \,\mu^n = \frac{1}{1 - 4d \gamma / (1 + \gamma)} < \infty,
\end{equation}
 therefore~$N_2$ is almost surely finite.
 Because, according to~\eqref{eq:GWBP-radius-1}, the expected number of~1s created by each~2 is bounded by~$4d$, the random variable~$N_1$ is almost surely finite as well.
 Labeling these~1s in the order they are created, letting~$Y_i$ be the time it takes for the~$i$th 1 to recover, and using that the random variables~$Y_i$ are almost surely finite and that
 $$ T_0 \leq Y_1 + Y_2 + \cdots + Y_{N_1}, $$
 we deduce that~$T_0$ is also almost surely finite.
\end{proof} \\ \\
 It directly follows from Lemma~\ref{lem:GWBP-time} that, starting with a single infected individual at the origin, the infection dies out globally.
 Proving local extinction requires more work.
 The basic idea is to use the exponential decay of the radius of the infection paths in Lemma~\ref{lem:GWBP-radius} to show that, regardless of the initial configuration, the probability of an infection path starting outside a large box centered at the origin and reaching the origin can be made arbitrarily small.
 Because this box is bounded, using that the time to extinction of each infection path is almost surely finite according to Lemma~\ref{lem:GWBP-time}, we also have that the infection paths starting inside the box die out in a finite time.
 In particular, the probability that the origin is infected tends to zero as time goes to infinity.
\begin{lemma}
\label{lem:GWBP-extinction}
 Let~$\beta_1 = 0$ and~$\gamma < 1 / (4d - 1)$. Then,
 $$ \begin{array}{l} \lim_{t \to \infty} P (\xi_t (0) \neq 0) = 0 \quad \hbox{for all} \quad \beta_2 \geq 0. \end{array} $$
\end{lemma}
\begin{proof}
 By monotonicity, it suffices to prove the result for~$\beta_2 = \infty$ and starting from the all~2 configuration.
 In this case, the infected region generated by the~2 initially at site~$y$ is dominated by its counterpart when starting with a single~2 at site~$y$ in an otherwise healthy population.
 In particular, letting~$\ep > 0$ be small, it follows from Lemma~\ref{lem:GWBP-radius} that the probability of an infection path starting from outside the cube~$\Lambda_r = [- r, r]^d$ and ending at the origin is bounded by
\begin{equation}
\label{eq:GWBP-extinction-1}
\begin{array}{l}
\displaystyle \sum_{i = 0}^{\infty} \ \Big(\vol \,(\Lambda_{r + i + 1}) - \vol \,(\Lambda_{r + i}) \Big) \bigg(\frac{4d \gamma}{1 + \gamma} \bigg)^{r + i} \vspace*{0pt} \\ \hspace*{75pt} \leq
\displaystyle \sum_{i = 0}^{\infty} \ 2d (2r + 2i + 3)^{d - 1} \bigg(\frac{4d \gamma}{1 + \gamma} \bigg)^{r + i} \vspace*{0pt} \\ \hspace*{150pt} =
\displaystyle 2d \ \sum_{j = r}^{\infty} \ (2j + 3)^{d - 1} \bigg(\frac{4d \gamma}{1 + \gamma} \bigg)^j \leq \ep / 2 \end{array}
\end{equation}
 for some~$r < \infty$ large fixed from now on.
 Let~$T_y$ be the extinction time of the infected cluster starting at site~$y$, and let~$\tau_r$ be the extinction time of the infected region generated by all the~2s initially in~$\Lambda_r$.
 The integer~$r$ being fixed, it follows from Lemma~\ref{lem:GWBP-time} that
 $$ \begin{array}{l}
      P (\tau_r < \infty) \geq P (\max_{y \in \Lambda_r} T_y < \infty) \geq P (\sum_{y \in \Lambda_r} T_y < \infty) = 1. \end{array} $$
 This shows that time~$\tau_r$ is almost surely finite, therefore
\begin{equation}
\label{eq:GWBP-extinction-2}
  P (\tau_r > T) \leq \ep / 2 \quad \hbox{for some} \quad T < \infty.
\end{equation}
 Combining~\eqref{eq:GWBP-extinction-1} which controls the infection paths starting from outside~$\Lambda_r$ and~\eqref{eq:GWBP-extinction-2} which controls the infection paths starting from inside~$\Lambda_r$, we deduce that the probability of the origin being infected at some time larger than~$T$ is bounded by~$\ep$, which proves the lemma.
\end{proof} \\ \\
 Because the evolution rules of the process are translation-invariant, the previous lemma also applies to each site of the lattice, which implies that, regardless of the initial configuration, the density of infected sites converges to zero, and shows the theorem.

%%%%%%%%%%%%%%%%%%%%%%%%%%%%%%%%%%%%%%%%%%%%%%%%%%%%%%%%%%%%%%%%%%%%%%%%%%%%%%%%%%%%%%%%%%%%%%%%%%%%%%%%%%%%%%%%%%%%%%%%%%%%%%%%%%%%%%%%%%%%%%%%%%%%%%%%%%%%%%%%%%%%%%%%%%%%%%%%%%%%%%%%%%%%%%%%%%

\section{Proof of Theorem~\ref{th:forest-fire}}
\label{sec:forest-fire}
 Recall that, under the assumptions of Theorem~\ref{th:forest-fire}, the rate~$\beta_2 = 0$, the rate~$\gamma$ is fixed and finite, and the process starts with a single~1 at the origin in an otherwise healthy population.
 To prove global survival, the idea is to use a construction inspired from Cox and Durrett~\cite{Cox_Durrett_1988} to show that the set of sites that ever get infected dominates a certain independent site percolation process that is supercritical when~$\beta_1$ is sufficiently large.
 To each site~$x$ and each oriented edge~$\vec{xy}$ where site~$y$ is a nearest neighbor of site~$x$, we attach independent random variables
 $$ \tau (x, y) = \exponential (\beta_1 / 2d) \quad \hbox{and} \quad \sigma (x) = \exponential (1 + \gamma). $$
 Then, we say that directed edge~$\vec{xy}$ is
 $$ \hbox{open when} \ \tau (x, y) < \sigma (x) \quad \hbox{and} \quad \hbox{closed when} \ \tau (x, y) \geq \sigma (x), $$
 and define the percolation cluster
 $$ C_0 = \{x \in \Z^d : \hbox{there is a directed path~$0 \to x$ of open edges} \}. $$
 Thinking of~$\tau (x, y)$ as the time it takes for a~1 at site~$x$ to spread germs to~$y$, and thinking of~$\sigma (x)$ as the time it takes for a~1 at site~$x$ to either recover, which occurs at rate one, or become symptomatic, which occurs at rate~$\gamma$, the presence of a directed open path~$0 \to x$ in the percolation process implies that the infection spreads up to site~$x$ in the interacting particle system.
 Using also the memoryless property of the exponential distribution, the epidemic model and the directed bond percolation process can be coupled in such a way that the sites that ever get infected contains the percolation cluster~$C_0$ from which we deduce the lower bound
\begin{equation}
\label{eq:forest-fire-1}
  P (\xi_t \neq \varnothing \ \hbox{for all} \ t) \geq P (|C_0| = \infty).
\end{equation}
 Therefore, to prove the theorem, it suffices to show that percolation occurs with positive probability when the infection rate~$\beta_1$ is large enough.
 Note that whether the directed edges starting from the same site~$x$ are open or closed are not independent events because they all depend on the realization of the random variable~$\sigma (x)$.
 However, declaring site~$x$ to be open if
 $$ \tau (x, y) < \sigma (x) \quad \hbox{for all} \quad y \in N_x $$
 and closed otherwise, whether different sites are open or closed are now independent events because they depend on disjoint collections of independent random variables.
 In addition, because a directed edge~$\vec{xy}$ is open whenever site~$x$ is open, we have
\begin{equation}
\label{eq:forest-fire-2}
\bar C_0 = \{x \in \Z^d : \hbox{there is a directed path~$0 \to x$ of open sites} \} \subset C_0.
\end{equation}
 In view of~\eqref{eq:forest-fire-1}--\eqref{eq:forest-fire-2}, to prove the theorem, it suffices to show that, for~$\beta_1$ large enough, sites are~(independently) open with a probability that exceeds the critical value~$p_c (\Z^d)$ of site percolation, which is known to be strictly less than one when~$d \geq 2$.
 Let
 $$ \tau_1, \tau_2, \ldots, \tau_{2d} = \exponential (\beta_1 / 2d) \quad \hbox{and} \quad \sigma = \exponential (1 + \gamma) $$
 be independent, and let~$\tau = \max (\tau_1, \tau_2, \ldots, \tau_{2d})$.
 Then, conditioning on the value of~$\sigma$, we obtain that the probability~$p (\beta_1, \gamma)$ of a given site being open is equal to
 $$ \begin{array}{rcl}
    \displaystyle P (\tau < \sigma) & \n = \n &
    \displaystyle \int_0^{\infty} P (\tau < t) \,(1 + \gamma) \,e^{- (1 + \gamma) t} \,dt \vspace*{6pt} \\ & \n = \n &
    \displaystyle \int_0^{\infty} (1 - e^{- \beta_1 t / 2d})^{2d} \,(1 + \gamma) \,e^{- (1 + \gamma) t} \,dt \vspace*{2pt} \\ & \n = \n &
    \displaystyle \int_0^{\infty} \sum_{k = 0}^{2d} {2d \choose k} (-1)^k (1 + \gamma) \,e^{- (1 + \gamma + k \beta_1 / 2d) t} \,dt =
    \displaystyle \sum_{k = 0}^{2d} {2d \choose k} \frac{(-1)^k}{1 + \frac{k \beta_1}{2d (1 + \gamma)}}. \end{array} $$
 For each~$\gamma$, the function~$p (\,\cdot \,, \gamma)$ is continuous, increasing on~$\R_+$, and
 $$ \begin{array}{l} p (0, \gamma) = 0 \qquad \hbox{and} \qquad \lim_{\beta_1 \to \infty} p (\beta_1, \gamma) = 1. \end{array} $$
 In particular, there exists a unique~$\bar \beta$ such that~$p (\bar \beta, \gamma) = p_c (\Z^d)$ and this~$\bar \beta$ is finite.
 In addition, it follows from~\eqref{eq:forest-fire-1}--\eqref{eq:forest-fire-2} and strict monotonicity that
 $$ P (\xi_t \neq \varnothing \ \hbox{for all} \ t) \geq P (|C_0| = \infty) \geq P (|\bar C_0| = \infty) > 0 \quad \hbox{for all} \quad \beta_1 > \bar \beta. $$
 This complete the proof of Theorem~\ref{th:forest-fire}.
 For instance, in the two-dimensional case, Kesten~\cite{Kesten_1980} proved that the critical value of undirected bond percolation is equal to one-half.
 This, together with a general result due to Grimmett and Stacey~\cite{Grimmett_Stacey_1998}, implies that
 $$ p_c (\Z^2) \leq 1 - \bigg(1 - \frac{1}{2} \bigg)^3 = \frac{7}{8}, $$
 and one can easily check that, for all~$r > 15.27$,
 $$ \sum_{k = 0}^4 {4 \choose k} \frac{(-1)^k}{1 + kr} = 1 - \frac{4}{1 + r} + \frac{6}{1 + 2r} - \frac{4}{1 + 3r} + \frac{1}{1 + 4r} > \frac{7}{8}. $$
 This shows global survival when~$\beta_1 > 4 \cdot 15.27 \cdot (1 + \gamma) = 61.08 \cdot (1 + \gamma)$. \\

\noindent {\bf Acknowledgments}.
 We thank two referees for mentioning reference~\cite{Deshayes_2014} we were not aware of, and for pointing out a mistake in a preliminary version of the proof of Lemma~\ref{lem:GWBP-radius}.

%%%%%%%%%%%%%%%%%%%%%%%%%%%%%%%%%%%%%%%%%%%%%%%%%%%%%%%%%%%%%%%%%%%%%%%%%%%%%%%%%%%%%%%%%%%%%%%%%%%%%%%%%%%%%%%%%%%%%%%%%%%%%%%%%%%%%%%%%%%%%%%%%%%%%%%%%%%%%%%%%%%%%%%%%%%%%%%%%%%%%%%%%%%%%%%%%%

\bibliographystyle{plain}
\bibliography{biblio.bib}

\end{document}

%% file: IPS0.pdf_t
\begin{picture}(0,0)%
\includegraphics{IPS0.pdf}%
\end{picture}%
\setlength{\unitlength}{3947sp}%
\begingroup\makeatletter\ifx\SetFigFont\undefined%
\gdef\SetFigFont#1#2#3#4#5{%
  \reset@font\fontsize{#1}{#2pt}%
  \fontfamily{#3}\fontseries{#4}\fontshape{#5}%
  \selectfont}%
\fi\endgroup%
\begin{picture}(15324,16053)(-11,-15202)
\put(3601,-6961){\makebox(0,0)[b]{\smash{{\SetFigFont{20}{24.0}{\familydefault}{\mddefault}{\updefault}{\color[rgb]{0,0,0}$\beta_1 = 8$, \ $\beta_2 = 0$, \ $\gamma = 1$}%
}}}}
\put(11701,-6961){\makebox(0,0)[b]{\smash{{\SetFigFont{20}{24.0}{\familydefault}{\mddefault}{\updefault}{\color[rgb]{0,0,0}$\beta_1 = 8$, \ $\beta_2 = 0$, \ $\gamma = 2$}%
}}}}
\put(3601,-15061){\makebox(0,0)[b]{\smash{{\SetFigFont{20}{24.0}{\familydefault}{\mddefault}{\updefault}{\color[rgb]{0,0,0}$\beta_1 = 0$, \ $\beta_2 = 4$, \ $\gamma = 1$}%
}}}}
\put(11701,-15061){\makebox(0,0)[b]{\smash{{\SetFigFont{20}{24.0}{\familydefault}{\mddefault}{\updefault}{\color[rgb]{0,0,0}$\beta_1 = 0$, \ $\beta_2 = 4$, \ $\gamma = 2$}%
}}}}
\end{picture}%

%% file: PS.pdf_t
\begin{picture}(0,0)%
\includegraphics{PS.pdf}%
\end{picture}%
\setlength{\unitlength}{3947sp}%
\begingroup\makeatletter\ifx\SetFigFont\undefined%
\gdef\SetFigFont#1#2#3#4#5{%
  \reset@font\fontsize{#1}{#2pt}%
  \fontfamily{#3}\fontseries{#4}\fontshape{#5}%
  \selectfont}%
\fi\endgroup%
\begin{picture}(13530,13746)(-914,80)
\put(12601,689){\makebox(0,0)[lb]{\smash{{\SetFigFont{25}{30.0}{\familydefault}{\mddefault}{\updefault}$\beta_1$}}}}
\put(  1,13439){\makebox(0,0)[b]{\smash{{\SetFigFont{25}{30.0}{\familydefault}{\mddefault}{\updefault}$\beta_2$}}}}
\put(7201,8189){\makebox(0,0)[b]{\smash{{\SetFigFont{25}{30.0}{\familydefault}{\mddefault}{\updefault}(\ref{eq:contact}.b)}}}}
\put(-299,12089){\makebox(0,0)[rb]{\smash{{\SetFigFont{25}{30.0}{\familydefault}{\mddefault}{\updefault}$\infty$}}}}
\put(10201,10589){\makebox(0,0)[b]{\smash{{\SetFigFont{25}{30.0}{\familydefault}{\mddefault}{\updefault}\eqref{eq:forest-fire}}}}}
\put(-899,9839){\makebox(0,0)[b]{\smash{{\SetFigFont{25}{30.0}{\familydefault}{\mddefault}{\updefault}$\gamma > \gamma_c$}}}}
\put(-899,9239){\makebox(0,0)[b]{\smash{{\SetFigFont{25}{30.0}{\familydefault}{\mddefault}{\updefault}(\ref{eq:gammac}.a)}}}}
\put(2101,239){\makebox(0,0)[b]{\smash{{\SetFigFont{25}{30.0}{\familydefault}{\mddefault}{\updefault}1}}}}
\put(-299,2789){\makebox(0,0)[rb]{\smash{{\SetFigFont{25}{30.0}{\familydefault}{\mddefault}{\updefault}1}}}}
\put(1051,1439){\makebox(0,0)[b]{\smash{{\SetFigFont{25}{30.0}{\familydefault}{\mddefault}{\updefault}(\ref{eq:contact}.a)}}}}
\put(7201,4439){\makebox(0,0)[b]{\smash{{\SetFigFont{25}{30.0}{\familydefault}{\mddefault}{\updefault}$\gamma = \infty$}}}}
\put(-299,3989){\makebox(0,0)[rb]{\smash{{\SetFigFont{25}{30.0}{\familydefault}{\mddefault}{\updefault}$\beta_c$}}}}
\put(3301,239){\makebox(0,0)[b]{\smash{{\SetFigFont{25}{30.0}{\familydefault}{\mddefault}{\updefault}$\beta_c$}}}}
\put(10201,1739){\makebox(0,0)[b]{\smash{{\SetFigFont{25}{30.0}{\familydefault}{\mddefault}{\updefault}\eqref{eq:forest-fire}}}}}
\put(601,12539){\makebox(0,0)[lb]{\smash{{\SetFigFont{25}{30.0}{\familydefault}{\mddefault}{\updefault}$\gamma < \gamma_c$ (\ref{eq:gammac}.b)}}}}
\put(3001,9689){\rotatebox{90.0}{\makebox(0,0)[b]{\smash{{\SetFigFont{25}{30.0}{\familydefault}{\mddefault}{\updefault}$\gamma = 0$}}}}}
\put(-299,5039){\makebox(0,0)[rb]{\smash{{\SetFigFont{25}{30.0}{\familydefault}{\mddefault}{\updefault}$\displaystyle \frac{1}{\gamma} + 1$}}}}
\put(4501,239){\makebox(0,0)[b]{\smash{{\SetFigFont{25}{30.0}{\familydefault}{\mddefault}{\updefault}$1 + \gamma$}}}}
\put(2701,6839){\makebox(0,0)[b]{\smash{{\SetFigFont{25}{30.0}{\familydefault}{\mddefault}{\updefault}(\ref{eq:top-left}.a)}}}}
\put(5401,3389){\makebox(0,0)[b]{\smash{{\SetFigFont{25}{30.0}{\familydefault}{\mddefault}{\updefault}(\ref{eq:bottom-right}.a)}}}}
\put(1351,4589){\makebox(0,0)[b]{\smash{{\SetFigFont{25}{30.0}{\familydefault}{\mddefault}{\updefault}(\ref{eq:top-left}.b)}}}}
\put(3901,2189){\makebox(0,0)[b]{\smash{{\SetFigFont{25}{30.0}{\familydefault}{\mddefault}{\updefault}(\ref{eq:bottom-right}.b)}}}}
\put(1051,2039){\makebox(0,0)[b]{\smash{{\SetFigFont{25}{30.0}{\rmdefault}{\bfdefault}{\updefault}extinction}}}}
\put(7201,8789){\makebox(0,0)[b]{\smash{{\SetFigFont{25}{30.0}{\rmdefault}{\bfdefault}{\updefault}survival}}}}
\end{picture}%

%% file: MF0.pdf_t
\begin{picture}(0,0)%
\includegraphics{MF0.pdf}%
\end{picture}%
\setlength{\unitlength}{3947sp}%
\begingroup\makeatletter\ifx\SetFigFont\undefined%
\gdef\SetFigFont#1#2#3#4#5{%
  \reset@font\fontsize{#1}{#2pt}%
  \fontfamily{#3}\fontseries{#4}\fontshape{#5}%
  \selectfont}%
\fi\endgroup%
\begin{picture}(15830,8826)(-114,-202)
\put(3601,-61){\makebox(0,0)[b]{\smash{{\SetFigFont{20}{24.0}{\familydefault}{\mddefault}{\updefault}{\color[rgb]{0,0,0}$\beta_1 = 2, \ \beta_2 = 3, \ \gamma = 2$}%
}}}}
\put(12001,-61){\makebox(0,0)[b]{\smash{{\SetFigFont{20}{24.0}{\familydefault}{\mddefault}{\updefault}{\color[rgb]{0,0,0}$\beta_1 = 2, \ \beta_2 = 3, \ \gamma = 0.5$}%
}}}}
\end{picture}%

%% file: GR.pdf_t
\begin{picture}(0,0)%
\includegraphics{GR.pdf}%
\end{picture}%
\setlength{\unitlength}{3947sp}%
\begingroup\makeatletter\ifx\SetFigFont\undefined%
\gdef\SetFigFont#1#2#3#4#5{%
  \reset@font\fontsize{#1}{#2pt}%
  \fontfamily{#3}\fontseries{#4}\fontshape{#5}%
  \selectfont}%
\fi\endgroup%
\begin{picture}(12624,12811)(-311,103)
\put(8251,3989){\makebox(0,0)[b]{\smash{{\SetFigFont{20}{24.0}{\familydefault}{\mddefault}{\updefault}1}}}}
\put(2251,7289){\makebox(0,0)[b]{\smash{{\SetFigFont{20}{24.0}{\familydefault}{\mddefault}{\updefault}2}}}}
\put(751,9689){\makebox(0,0)[b]{\smash{{\SetFigFont{20}{24.0}{\familydefault}{\mddefault}{\updefault}2}}}}
\put(3751,8789){\makebox(0,0)[b]{\smash{{\SetFigFont{20}{24.0}{\familydefault}{\mddefault}{\updefault}1}}}}
\put(2251,5789){\makebox(0,0)[b]{\smash{{\SetFigFont{20}{24.0}{\familydefault}{\mddefault}{\updefault}1}}}}
\put(751,2789){\makebox(0,0)[b]{\smash{{\SetFigFont{20}{24.0}{\familydefault}{\mddefault}{\updefault}1}}}}
\put(6751,10889){\makebox(0,0)[b]{\smash{{\SetFigFont{20}{24.0}{\familydefault}{\mddefault}{\updefault}1}}}}
\put(11251,5489){\makebox(0,0)[b]{\smash{{\SetFigFont{20}{24.0}{\familydefault}{\mddefault}{\updefault}2}}}}
\put(5251,4589){\makebox(0,0)[b]{\smash{{\SetFigFont{20}{24.0}{\familydefault}{\mddefault}{\updefault}1}}}}
\put(9751,7889){\makebox(0,0)[b]{\smash{{\SetFigFont{20}{24.0}{\familydefault}{\mddefault}{\updefault}1}}}}
\put(6751,6089){\makebox(0,0)[b]{\smash{{\SetFigFont{20}{24.0}{\familydefault}{\mddefault}{\updefault}2}}}}
\put(11251,9389){\makebox(0,0)[b]{\smash{{\SetFigFont{20}{24.0}{\familydefault}{\mddefault}{\updefault}2}}}}
\put(12001,239){\makebox(0,0)[b]{\smash{{\SetFigFont{25}{30.0}{\familydefault}{\mddefault}{\updefault}4}}}}
\put(10501,239){\makebox(0,0)[b]{\smash{{\SetFigFont{25}{30.0}{\familydefault}{\mddefault}{\updefault}3}}}}
\put(9001,239){\makebox(0,0)[b]{\smash{{\SetFigFont{25}{30.0}{\familydefault}{\mddefault}{\updefault}2}}}}
\put(7501,239){\makebox(0,0)[b]{\smash{{\SetFigFont{25}{30.0}{\familydefault}{\mddefault}{\updefault}1}}}}
\put(6001,239){\makebox(0,0)[b]{\smash{{\SetFigFont{25}{30.0}{\familydefault}{\mddefault}{\updefault}0}}}}
\put(4501,239){\makebox(0,0)[b]{\smash{{\SetFigFont{25}{30.0}{\familydefault}{\mddefault}{\updefault}$-1$}}}}
\put(3001,239){\makebox(0,0)[b]{\smash{{\SetFigFont{25}{30.0}{\familydefault}{\mddefault}{\updefault}$-2$}}}}
\put(1501,239){\makebox(0,0)[b]{\smash{{\SetFigFont{25}{30.0}{\familydefault}{\mddefault}{\updefault}$-3$}}}}
\put(  1,239){\makebox(0,0)[b]{\smash{{\SetFigFont{25}{30.0}{\familydefault}{\mddefault}{\updefault}$-4$}}}}
\end{picture}%

%% file: epidemic-3.bbl
\begin{thebibliography}{10}

\bibitem{Bezuidenhout_Grimmett_1990}
C.~Bezuidenhout and G.~Grimmett.
\newblock The critical contact process dies out.
\newblock {\em Ann. Probab.}, 18(4):1462--1482, 1990.

\bibitem{Bezuidenhout_Grimmett_1991}
C.~Bezuidenhout and G.~Grimmett.
\newblock Exponential decay for subcritical contact and percolation processes.
\newblock {\em Ann. Probab.}, 19(3):984--1009, 1991.

\bibitem{Cox_Durrett_1988}
J.~T. Cox and R.~Durrett.
\newblock Limit theorems for the spread of epidemics and forest fires.
\newblock {\em Stochastic Process. Appl.}, 30(2):171--191, 1988.

\bibitem{Deshayes_2014}
A.~Deshayes.
\newblock The contact process with aging.
\newblock {\em ALEA Lat. Am. J. Probab. Math. Stat.}, 11(2):845--883, 2014.

\bibitem{Durrett_Neuhauser_1991}
R.~Durrett and C.~Neuhauser.
\newblock Epidemics with recovery in {$D=2$}.
\newblock {\em Ann. Appl. Probab.}, 1(2):189--206, 1991.

\bibitem{Foxall_2015}
E.~Foxall.
\newblock New results for the two-stage contact process.
\newblock {\em J. Appl. Probab.}, 52(1):258--268, 2015.

\bibitem{Grimmett_Stacey_1998}
G.~Grimmett and A.~Stacey.
\newblock Critical probabilities for site and bond percolation models.
\newblock {\em Ann. Probab.}, 26(4):1788--1812, 1998.

\bibitem{Harris_1972}
T.~E. Harris.
\newblock Nearest-neighbor {M}arkov interaction processes on multidimensional
  lattices.
\newblock {\em Advances in Math.}, 9:66--89, 1972.

\bibitem{Harris_1974}
T.~E. Harris.
\newblock Contact interactions on a lattice.
\newblock {\em Ann. Probability}, 2:969--988, 1974.

\bibitem{Harris_1978}
T.~E. Harris.
\newblock Additive set-valued {M}arkov processes and graphical methods.
\newblock {\em Ann. Probability}, 6(3):355--378, 1978.

\bibitem{Kesten_1980}
H.~Kesten.
\newblock The critical probability of bond percolation on the square lattice
  equals {${1\over 2}$}.
\newblock {\em Comm. Math. Phys.}, 74(1):41--59, 1980.

\bibitem{Krone_1999}
S.~M. Krone.
\newblock The two-stage contact process.
\newblock {\em Ann. Appl. Probab.}, 9(2):331--351, 1999.

\bibitem{Lanchier_2024}
N.~Lanchier.
\newblock {\em Stochastic interacting systems in life and social sciences}.
\newblock To appear in De Gruyter Series in Probability and Stochastics. De
  Gruyter, Berlin, 2024.

\bibitem{Liggett_1985}
T.~M. Liggett.
\newblock {\em Interacting particle systems}, volume 276 of {\em Grundlehren
  der mathematischen Wissenschaften [Fundamental Principles of Mathematical
  Sciences]}.
\newblock Springer-Verlag, New York, 1985.

\bibitem{Liggett_1999}
T.~M. Liggett.
\newblock {\em Stochastic interacting systems: contact, voter and exclusion
  processes}, volume 324 of {\em Grundlehren der mathematischen Wissenschaften
  [Fundamental Principles of Mathematical Sciences]}.
\newblock Springer-Verlag, Berlin, 1999.

\bibitem{vandenBerg_Grimmett_Schinazi_1998}
J.~van~den Berg, G.~R. Grimmett, and R.~B. Schinazi.
\newblock Dependent random graphs and spatial epidemics.
\newblock {\em Ann. Appl. Probab.}, 8(2):317--336, 1998.

\end{thebibliography}
